\documentclass[12pt,a4paper]{amsart}
\usepackage{fancyhdr,a4wide,fontenc}
\usepackage[dvipsnames]{xcolor}
\usepackage{mathtools}

\usepackage{amsmath}
\usepackage{amssymb}
\usepackage{amsthm}
\usepackage{float}
\usepackage{caption}
\usepackage[caption = false]{subfig}
\usepackage{enumerate}
\usepackage{fullpage}
\usepackage{graphicx}
\usepackage{multicol,calc}
\usepackage{color}
\usepackage{pifont}
\usepackage{comment}

\usepackage{tikz}   
\usetikzlibrary{arrows.meta,calc} 

\usepackage{hyperref,lscape,appendix}

\input xy
\xyoption{all}

\usepackage[all]{xy}

\newcommand{\bfb}{{\bf b}}

\newcommand{\bfz}{{\bf z}}

\newcommand{\calQ}{{\mathcal Q}}
\newcommand{\calD}{{\mathcal D}}
\newcommand{\calM}{{\mathcal M}}
\newcommand{\loops}{{\mathcal{L}}}

\newcommand{\Sc}[1]{\overline{{\mathcal S}_0}}

\newcommand{\trans}{^\top}
\newcommand{\bR}{{\mathbb R}}
\newcommand{\bN}{{\mathbb N}}

\DeclareMathOperator{\vecop}{vec}

\DeclareMathOperator{\spec}{spec}

\newcommand\dvec[1]{#1}

 \usepackage{graphicx}

\usetikzlibrary{positioning}

  \tikzset{main node/.style={circle,fill=red!10,draw,minimum size=0.2cm,inner sep=0pt},
             }
\tikzset{
    every node/.style={draw, circle, fill=white, inner sep=1pt}
    }
            
\newtheorem{theorem}{Theorem}[section]
\newtheorem{corollary}[theorem]{Corollary}

\newtheorem{lemma}[theorem]{Lemma}

\newtheorem{question}[theorem]{Question}
\theoremstyle{definition}
\newtheorem{example}[theorem]{Example}
\newtheorem{definition}[theorem]{Definition}
\newtheorem{remark}[theorem]{Remark}

\newtheorem{problem}[theorem]{Problem}

\title{Combinatorial aspects of the non-symmetric strong spectral property for graphs}

\author{Sara Koljan\v{c}i\'{c}}

\address[S. Koljan\v{c}i\'{c}]{Faculty of Natural Sciences and Mathematics, University of Banja Luka, Bosnia and Herzegovina}
\email{sara.jevdjenic@pmf.unibl.org}

\author{Polona Oblak}

\address[P.~Oblak]{Faculty of Computer and Information Science and Faculty of Mathematics and Physics, University of Ljubljana, Slovenia; and Institute of Mathematics, Physics, and Mechanics, Slovenia}
\email{polona.oblak@fri.uni-lj.si}

\thanks{Both authors acknowledge the financial support from the bilateral projects no.~BI-BA/24-25-024,~BI-BA/26-27-007 and~BI-BA/26-27-008 funded by Slovenian Research Agency and Ministry of Civil Affairs, Bosnia and Herzegovina. Polona Oblak has also received funding from Slovenian Research Agency, research core funding no.~P1-0222.}

\date{\today}

\begin{document}

\maketitle
\begin{abstract}

In this paper, we investigate the non-symmetric Strong Spectral Property (nSSP) from a combinatorial perspective. To zero–nonzero patterns of matrices we associate directed graphs and study when they require or allow the nSSP, providing a framework that avoids verifying the nSSP for individual matrices. 

 A new combinatorial method is introduced and used to recognise several patterns that require the nSSP. It is shown that loop assignments in double paths play a critical role in establishing this property, and we show that an open question regarding irreducible tridiagonal patterns has a negative answer. We also investigate whether the minimum number of arcs in a directed graph on $n$ vertices that requires the nSSP, is equal to $2n-1$, and confirm this minimum for several specific digraph families. 

\end{abstract}



\bigskip

\section{Introduction}

The Inverse Eigenvalue Problems (IEP) are central questions in spectral theory, concerned with determining matrices of a specific structure that realise a prescribed set of eigenvalues. The continuing difficulty of these problems has sustained a large and active research community, driven by multidisciplinary applications,~\cite{2005-IEP-Chu-Golub}. In the last decade, the Strong Properties of matrices, which rely on the transversal intersection of manifolds and the Implicit Function Theorem, have proven to be a powerful tool for the symmetric inverse eigenvalue problem, and have significantly advanced the solutions for the IEP for graphs,~\cite{Barrett17GeneralizationsSAP, 1993-Verdiere-SAP, Hogben-Lin-Shader-IEPG-book}. Recently, their non-symmetric versions were introduced, see e.g.~\cite{Arav-24-nSSP,2025-Cheung-Shader-SSVP,2025-Curtis-nSMP,2020-Curtis-Shader-SIPP,Fallat22Bifurcation,2025-Saha-Tilis-VanderMeulen-VanTuyl-nSMP}.
In this work we investigate the non-symmetric Strong Spectral Property (nSSP) of matrices, firstly introduced in~\cite{Fallat22Bifurcation}. 

\begin{definition}[\cite{Fallat22Bifurcation}]
 A matrix $A\in \mathbb{R}^{n\times n}$ has the \emph{non-symmetric Strong Spectral Property (nSSP)} provided $X=O$ is the only matrix such that $A\circ X=[A,X\trans]=O.$
 \end{definition}
 
 The authors in~\cite{Fallat22Bifurcation} developed the non-symmetric  
 Superpattern Lemma, which shows the strength of an existence of a matrix with a desired spectrum and having the nSSP.

\begin{lemma}[Superpattern Lemma,~\cite{Fallat22Bifurcation}]\label{lem:superpattern-lemma}
    If an $n\times n$ real matrix $A$ of pattern $P$ has the nSSP, then for every superpattern $P'$ of $P$ there exists  a matrix $A'$ of pattern $P'$ with the nSSP, which is similar to $A$.
\end{lemma}

 Existing literature concerning the IEP for non-symmetric matrices has predominantly concentrated on tridiagonal patterns, see e.g., \cite{2024-IEP-nonsymmetric-paths, 1997-IEP-examples}. However, Superpattern Lemma together with the Bifurcation Lemma (see~\cite[Theorem 5.3]{Fallat22Bifurcation}) enable researchers to advance the studies in several subproblems of the IEP ,~\cite{Arav-24-nSSP,2026-Berliner-Refined-inertia, 2022-Breen-q-for-sign-pattern, 2024-Breen-Allow-sequence, 2026-LAA-Cavers-VanderMeulen}.

It is difficult to resolve the existence of a matrix with the nSSP and with a desired spectrum.  We follow~\cite{20-Lin-SSPgraph} and give a combinatorial description of non-symmetric patterns for which every matrix has the nSSP. Such pattern proved to be useful in IEP, since we do not have to verify if a matrix $A$ in the Superpattern Lemma has the nSSP. In particular, it says that all spectra achieved by  pattern $P$, in which every matrix has the nSSP, can be also realised in every superpattern of $P$.  

In this paper, the zero--nonzero patterns will be presented by directed graphs with loops. 
For $A=(a_{ij})\in \bR^{n\times n}$ we denote by $\calD(A)$ the directed graph $G=(V(G),E(G))$, for which $V(G)=[n]$ and $a_{ij}\ne 0$ if and only if $(i,j)\in E(G)$. For the fundamental definitions and properties of these digraphs, the reader is referred to Subsection~\ref{sec:notation}.
For a directed graph $G$ with $n$ vertices, let $\calM(G)$ be the set of all $n\times n$ real matrices $A$ with $\calD(A)=G$. 

\begin{definition}
    If every matrix $A \in \calM(G)$ has the nSSP, then we say that $G$ \emph{requires the nSSP}, and if there is a matrix $A \in \calM(G)$ with the nSSP, then $G$ \emph{allows the nSSP}.
\end{definition}

A similar notion was introduced in~\cite{2025-Curtis-nSMP}, where the authors define the nSMP and investigate digraphs and sign patterns that require the nSMP. It was proved in~\cite
{2025-Saha-Tilis-VanderMeulen-VanTuyl-nSMP} 
that a pattern 
allows the nSMP if and only if it allows distinct eigenvalues. Some additional work on the number of distinct eigenvalues of a pattern has also been carried out in~\cite{2022-Breen-q-for-sign-pattern,2024-Breen-Allow-sequence,2025-Curtis-nSMP,2025-Saha-Tilis-VanderMeulen-VanTuyl-nSMP}. 
Note that any digraph requiring the nSSP necessarily requires the nSMP. Consequently, several results established herein for the nSSP hold in the context of the nSMP as well.

In~\cite{Arav-24-nSSP} the authors showed that nSSP is equivalent to Similarity-Transversality Property (STP), which relies on the rank of the Jacobian matrix of the commutator $[A,X]$.  
The authors present several zero--nonzero patterns and several sign patterns that require the nSSP. Also, the authors observe that the loop assignment plays a crucial role in a digraph requiring/allowing the nSSP.

\begin{lemma}\label{lem:noloops-all-loops-nssp}
Let $G$ be a directed graph. 
\begin{enumerate}[(A)]
    \item \cite[Observation 3.4]{Arav-24-nSSP} If $G$ has no loops, then $G$ does not allow the nSSP.
    \item \cite[Proposition 5.4]{Arav-24-nSSP}\label{all-loops} If all vertices of $G$ have loops, then $G$ allows the nSSP.
\end{enumerate}
\end{lemma}

In our work, we examine more closely how the assignment of the loops influences whether a digraph possesses the nSSP. Our work is motivated by two of their open questions. We present them in the next two subsections and give an overview of our work.

\subsection{Irreducible tridiagonal patterns} 
In~\cite[Problem~6.5]{Arav-24-nSSP}, the authors ask whether every irreducible tridiagonal zero--nonzero pattern with at least one nonzero diagonal entry requires the STP. Note that for an irreducible tridiagonal matrix $A$ its directed graph ${\mathcal D}(A)$ is a strongly connected directed graph, whose underlying simple graph is a path. We call it a double path, and we label vertices of the double path in its natural labeling as following.

\begin{definition}
   A digraph $\dvec{P}_{n,\loops}$ is called \emph{a double path  with loop assignment $\loops\subseteq [n]$} if $V(\dvec{P}_{n,\loops})=[n]$ and $$E(\dvec{P}_{n,\loops})=\{(i,i+1),(i+1,i),(\ell,\ell) \colon i\in [n-1],\ell \in \loops\}.$$ 
\end{definition}

Note that a matrix $A=(a_{ij})\in {\mathcal M}(\dvec{P}_{n,\loops})$ has $a_{ij}\ne 0$ if and only if $i=j$ and $i\in \loops$ or $|i-j|=1$. This means $A$ is an irreducible tridiagonal matrix with nonzero diagonal entries in the position of $\loops$.

It is straightforward to check that  matrices 
$$A=\begin{pmatrix}
a & b & 0 \\
 c & 0 & c \\
 0 & b & a
\end{pmatrix} \quad \text{and } \quad
X=\begin{pmatrix}
 0 & 0 & 1 \\
 0 & 1 & 0 \\
 1 & 0 & 0 
\end{pmatrix}
$$
satisfy $A\circ X=[A,X\trans]=O$. So, the digraph ${\mathcal D}(A)=\dvec{P}_{3,\{1,3\}}$ shown on Figure~\ref{fig:path-3+2-loops},  does not require the nSSP, and hence the answer to~\cite[Problem~6.5]{Arav-24-nSSP} is negative. 

 \begin{figure}[h]
\begin{tikzpicture}
\foreach \i in {1,...,3} {
    \node[fill=white] (\i) at (\i,0) {};
   
};

\tikzset{every loop/.style={min distance=4mm,in=120,out=60,looseness=40}}
\foreach \i in {1,3} {
     \draw[thick,->] (\i) to[loop above] (\i);
};
 
\draw[->, thick] (1) to[bend left] (2);
\draw[->, thick] (2) to[bend left] (1);
\draw[->, thick] (3) to[bend left] (2);
\draw[->, thick] (2) to[bend left] (3);

\node[rectangle, draw=none] at (2,-1) {$\dvec{P}_{3,\{1,3\}}$};
 \end{tikzpicture}
 
\caption{Example of a directed graph which does not require but allows the nSSP, see also Theorem~\ref{thm:double_paths-first-m-loops} together with Lemma~\ref{lem:superpattern-lemma}.}\label{fig:path-3+2-loops}
\end{figure}

In this paper we show that for any $n\geq 3$, $\dvec{P}_{n,\{1,n\}}$ allows the nSSP but does not require it (see Theorem~\ref{thm:double-path-vtx-2}).
This example motivates us to
propose the following alternative refinement to~\cite[Problem~6.5]{Arav-24-nSSP}, which we will investigate in Section~\ref{sec:double-paths}:

\begin{problem} \label{problem-6.5}  For every positive integer $n$ classify loop assignments $\loops\subseteq [n]$ of the double path on $n$ vertices so that $\dvec{P}_{n,\loops}$ 
\begin{enumerate}[(a)] 
\item requires the nSSP,
\item does not require but allows the nSSP,
\item does not allow the nSSP.
\end{enumerate}
\end{problem}

The difficulty of this problem is illustrated by Theorem~\ref{thm:double-path-vtx-2} and Example~\ref{ex:central_being_looped}, in which we show that there are families of double paths with a single loop that require the nSSP, families that do not require but still allow the nSSP, and families that do not allow the nSSP.
In Theorem~\ref{thm:double_paths-first-m-loops}, we show that the nSSP is required by all double paths having loops only on an initial segment of vertices.

All the results in the Section~\ref{sec:double-paths} rely on the powerful combinatorial rule which we introduce in Lemma~\ref{lem:rule1}, and is an analogue of the monomial rule for the symmetric Strong Spectral Property,~\cite[Lemma~3.1]{20-Lin-SSPgraph}.  Its strength is established in Theorem~\ref{thm:rule1-nSSP}, which is used extensively throughout the Section ~\ref{sec:double-paths}. These findings provide the basis for resolving Problem~\ref{problem-6.5} for $n \leq 5$, as demonstrated in Example~\ref{ex:n<=5} and summarised in Table~\ref{tab:small-paths}.

\begin{table}[htb]
    \centering
    \begin{tabular}{c||c|c|c||}
         $n$&  require&  allow  &   not allow  \\
         & &   not require   & \\
     \hline
     \hline
        $n=1$ & 
         \begin{tikzpicture}[scale=0.5]
         \node[fill=white] (1) at (0,-0.5) {};
          \tikzset{every loop/.style={min distance=4mm,in=120,out=60,looseness=20}}
           \foreach \i in {1} {
             \draw[thick,->] (\i) to[loop above] (\i);
        };\end{tikzpicture}
        &&
        \begin{tikzpicture}[scale=0.5] 
             \node[fill=white] (1) at (0,-0.5) {};
        \end{tikzpicture}\\
        \hline
    $n=2$ & 
         \begin{tikzpicture}[scale=0.5] 
         \node[fill=white] (1) at (0,-0.5) {};
         \node[fill=white] (2) at (1,-0.5) {};
          \tikzset{every loop/.style={min distance=4mm,in=120,out=60,looseness=20}}
           \foreach \i in {1} {
             \draw[thick,->] (\i) to[loop above] (\i);
         };
            \draw[thick](1)-- (2);
        \end{tikzpicture} \quad
          \begin{tikzpicture}[scale=0.5] 
         \node[fill=white] (1) at (0,-0.5) {};
         \node[fill=white] (2) at (1,-0.5) {};
          \tikzset{every loop/.style={min distance=4mm,in=120,out=60,looseness=20}}
           \foreach \i in {1,2} {
             \draw[thick,->] (\i) to[loop above] (\i);
         };
            \draw[thick](1)-- (2);
        \end{tikzpicture}
        &&
          \begin{tikzpicture}[scale=0.5] 
         \node[fill=white] (1) at (0,-0.5) {};
         \node[fill=white] (2) at (1,-0.5) {};
            \draw[thick](1)-- (2);
        \end{tikzpicture}\\
        \hline
$n=3$ & 
         \begin{tikzpicture}[scale=0.5] 
         \node[fill=white] (1) at (0,-0.5) {};
         \node[fill=white] (2) at (1,-0.5) {};
          \node[fill=white] (3) at (2,-0.5) {};
          \tikzset{every loop/.style={min distance=4mm,in=120,out=60,looseness=20}}
           \foreach \i in {1} {
             \draw[thick,->] (\i) to[loop above] (\i);
         };
            \draw[thick](1)-- (2)--(3);
        \end{tikzpicture} \quad
          \begin{tikzpicture}[scale=0.5] 
         \node[fill=white] (1) at (0,-0.5) {};
         \node[fill=white] (2) at (1,-0.5) {};
          \node[fill=white] (3) at (2,-0.5) {};
          \tikzset{every loop/.style={min distance=4mm,in=120,out=60,looseness=20}}
           \foreach \i in {1,2} {
             \draw[thick,->] (\i) to[loop above] (\i);
         };
            \draw[thick](1)-- (2)--(3);
        \end{tikzpicture}\quad
        \begin{tikzpicture}[scale=0.5] 
         \node[fill=white] (1) at (0,-0.5) {};
         \node[fill=white] (2) at (1,-0.5) {};
          \node[fill=white] (3) at (2,-0.5) {};
          \tikzset{every loop/.style={min distance=4mm,in=120,out=60,looseness=20}}
           \foreach \i in {1,2,3} {
             \draw[thick,->] (\i) to[loop above] (\i);
         };
            \draw[thick](1)-- (2)--(3);
        \end{tikzpicture}
        &\begin{tikzpicture}[scale=0.5] 
         \node[fill=white] (1) at (0,-0.5) {};
         \node[fill=white] (2) at (1,-0.5) {};
          \node[fill=white] (3) at (2,-0.5) {};
          \tikzset{every loop/.style={min distance=4mm,in=120,out=60,looseness=20}}
           \foreach \i in {1,3} {
             \draw[thick,->] (\i) to[loop above] (\i);
         };
            \draw[thick](1)-- (2)--(3);
        \end{tikzpicture} 
        &
          \begin{tikzpicture}[scale=0.5] 
         \node[fill=white] (1) at (0,-0.5) {};
         \node[fill=white] (2) at (1,-0.5) {};
          \node[fill=white] (3) at (2,-0.5) {};
            \draw[thick](1)-- (2)--(3);
        \end{tikzpicture}\quad
        \begin{tikzpicture}[scale=0.5]
         \node[fill=white] (1) at (0,-0.5) {};
         \node[fill=white] (2) at (1,-0.5) {};
          \node[fill=white] (3) at (2,-0.5) {};
          \tikzset{every loop/.style={min distance=4mm,in=120,out=60,looseness=20}}
           \foreach \i in {2} {
             \draw[thick,->] (\i) to[loop above] (\i);
         };
            \draw[thick](1)-- (2)--(3);
        \end{tikzpicture} \\
        \hline
 $n=4$ & 
         \begin{tikzpicture}[scale=0.5] 
         \node[fill=white] (1) at (0,-0.5) {};
         \node[fill=white] (2) at (1,-0.5) {};
          \node[fill=white] (3) at (2,-0.5) {};
            \node[fill=white] (4) at (3,-0.5) {};
          \tikzset{every loop/.style={min distance=4mm,in=120,out=60,looseness=20}}
           \foreach \i in {1} {
             \draw[thick,->] (\i) to[loop above] (\i);
         };
            \draw[thick](1)-- (2)--(3)--(4);
        \end{tikzpicture} \quad
          \begin{tikzpicture}[scale=0.5] 
         \node[fill=white] (1) at (0,-0.5) {};
         \node[fill=white] (2) at (1,-0.5) {};
          \node[fill=white] (3) at (2,-0.5) {};
                  \node[fill=white] (4) at (3,-0.5) {};
          \tikzset{every loop/.style={min distance=4mm,in=120,out=60,looseness=20}}
           \foreach \i in {2} {
             \draw[thick,->] (\i) to[loop above] (\i);
         };
            \draw[thick](1)-- (2)--(3)--(4);
        \end{tikzpicture}
        &\begin{tikzpicture}[scale=0.5] 
         \node[fill=white] (1) at (0,-0.5) {};
         \node[fill=white] (2) at (1,-0.5) {};
          \node[fill=white] (3) at (2,-0.5) {};
         \node[fill=white] (4) at (3,-0.5) {};
          \tikzset{every loop/.style={min distance=4mm,in=120,out=60,looseness=20}}
           \foreach \i in {1,3} {
             \draw[thick,->] (\i) to[loop above] (\i);
         };
            \draw[thick](1)-- (2)--(3)--(4);
        \end{tikzpicture}  
        &    \begin{tikzpicture}[scale=0.5] 
         \node[fill=white] (1) at (0,-0.5) {};
         \node[fill=white] (2) at (1,-0.5) {};
          \node[fill=white] (3) at (2,-0.5) {};
         \node[fill=white] (4) at (3,-0.5) {};
            \draw[thick](1)-- (2)--(3)--(4);
        \end{tikzpicture} \\
        & \begin{tikzpicture}[scale=0.5] 
         \node[fill=white] (1) at (0,-0.5) {};
         \node[fill=white] (2) at (1,-0.5) {};
          \node[fill=white] (3) at (2,-0.5) {};
         \node[fill=white] (4) at (3,-0.5) {};
          \tikzset{every loop/.style={min distance=4mm,in=120,out=60,looseness=20}}
           \foreach \i in {1,2} {
             \draw[thick,->] (\i) to[loop above] (\i);
         };
            \draw[thick](1)-- (2)--(3)--(4);
        \end{tikzpicture}\quad
         \begin{tikzpicture}[scale=0.5] 
         \node[fill=white] (1) at (0,-0.5) {};
         \node[fill=white] (2) at (1,-0.5) {};
          \node[fill=white] (3) at (2,-0.5) {};
            \node[fill=white] (4) at (3,-0.5) {};
          \tikzset{every loop/.style={min distance=4mm,in=120,out=60,looseness=20}}
           \foreach \i in {1,2,3} {
             \draw[thick,->] (\i) to[loop above] (\i);
         };
            \draw[thick](1)-- (2)--(3)--(4);
        \end{tikzpicture}         
        &
        \begin{tikzpicture}[scale=0.5] 
         \node[fill=white] (1) at (0,-0.5) {};
         \node[fill=white] (2) at (1,-0.5) {};
          \node[fill=white] (3) at (2,-0.5) {};
         \node[fill=white] (4) at (3,-0.5) {};
          \tikzset{every loop/.style={min distance=4mm,in=120,out=60,looseness=20}}
           \foreach \i in {1,4} {
             \draw[thick,->] (\i) to[loop above] (\i);
         };
            \draw[thick](1)-- (2)--(3)--(4);
        \end{tikzpicture}
        &
         \\
         &  \begin{tikzpicture}[scale=0.5] 
         \node[fill=white] (1) at (0,-0.5) {};
         \node[fill=white] (2) at (1,-0.5) {};
          \node[fill=white] (3) at (2,-0.5) {};
                  \node[fill=white] (4) at (3,-0.5) {};
          \tikzset{every loop/.style={min distance=4mm,in=120,out=60,looseness=20}}
           \foreach \i in {1,2,4} {
             \draw[thick,->] (\i) to[loop above] (\i);
         };
            \draw[thick](1)-- (2)--(3)--(4);
        \end{tikzpicture}
        \quad
          \begin{tikzpicture}[scale=0.5] 
         \node[fill=white] (1) at (0,-0.5) {};
         \node[fill=white] (2) at (1,-0.5) {};
          \node[fill=white] (3) at (2,-0.5) {};
         \node[fill=white] (4) at (3,-0.5) {};
          \tikzset{every loop/.style={min distance=4mm,in=120,out=60,looseness=20}}
           \foreach \i in {1,2,3,4} {
             \draw[thick,->] (\i) to[loop above] (\i);
         };
            \draw[thick](1)-- (2)--(3)--(4);
        \end{tikzpicture}&&\\
        \hline
 $n=5$ & 
        \begin{tikzpicture}[scale=0.5]
         \node[fill=white] (1) at (0,-0.5) {};
         \node[fill=white] (2) at (1,-0.5) {};
          \node[fill=white] (3) at (2,-0.5) {};
            \node[fill=white] (4) at (3,-0.5) {};
            \node[fill=white] (5) at (4,-0.5) {};
          \tikzset{every loop/.style={min distance=4mm,in=120,out=60,looseness=20}}
           \foreach \i in {1} {
             \draw[thick,->] (\i) to[loop above] (\i);
         };
            \draw[thick](1)-- (2)--(3)--(4)--(5);
        \end{tikzpicture} \quad
          \begin{tikzpicture}[scale=0.5] 
         \node[fill=white] (1) at (0,-0.5) {};
         \node[fill=white] (2) at (1,-0.5) {};
          \node[fill=white] (3) at (2,-0.5) {};
                  \node[fill=white] (4) at (3,-0.5) {};\node[fill=white] (5) at (4,-0.5) {};
          \tikzset{every loop/.style={min distance=4mm,in=120,out=60,looseness=20}}
           \foreach \i in {1,2} {
             \draw[thick,->] (\i) to[loop above] (\i);
         };
            \draw[thick](1)-- (2)--(3)--(4)--(5);
        \end{tikzpicture}
        &\begin{tikzpicture}[scale=0.5] 
         \node[fill=white] (1) at (0,-0.5) {};
         \node[fill=white] (2) at (1,-0.5) {};
          \node[fill=white] (3) at (2,-0.5) {};
         \node[fill=white] (4) at (3,-0.5) {};
         \node[fill=white] (5) at (4,-0.5) {};
          \tikzset{every loop/.style={min distance=4mm,in=120,out=60,looseness=20}}
           \foreach \i in {3} {
             \draw[thick,->] (\i) to[loop above] (\i);
         };
            \draw[thick](1)-- (2)--(3)--(4)--(5);
        \end{tikzpicture}  
        &
          \begin{tikzpicture}[scale=0.5] 
         \node[fill=white] (1) at (0,-0.5) {};
         \node[fill=white] (2) at (1,-0.5) {};
          \node[fill=white] (3) at (2,-0.5) {};
         \node[fill=white] (4) at (3,-0.5) {};
         \node[fill=white] (5) at (4,-0.5) {};
            \draw[thick](1)-- (2)--(3)--(4)--(5);
        \end{tikzpicture} \\
        &      \begin{tikzpicture}[scale=0.5] 
         \node[fill=white] (1) at (0,-0.5) {};
         \node[fill=white] (2) at (1,-0.5) {};
          \node[fill=white] (3) at (2,-0.5) {};
         \node[fill=white] (4) at (3,-0.5) {};
         \node[fill=white] (5) at (4,-0.5) {};
          \tikzset{every loop/.style={min distance=4mm,in=120,out=60,looseness=20}}
           \foreach \i in {1,2,3} {
             \draw[thick,->] (\i) to[loop above] (\i);
         };
            \draw[thick](1)-- (2)--(3)--(4)--(5);
        \end{tikzpicture}\quad
         \begin{tikzpicture}[scale=0.5] 
         \node[fill=white] (1) at (0,-0.5) {};
         \node[fill=white] (2) at (1,-0.5) {};
          \node[fill=white] (3) at (2,-0.5) {};
            \node[fill=white] (4) at (3,-0.5) {};
            \node[fill=white] (5) at (4,-0.5) {};
          \tikzset{every loop/.style={min distance=4mm,in=120,out=60,looseness=20}}
           \foreach \i in {1,2,3,4} {
             \draw[thick,->] (\i) to[loop above] (\i);
         };
            \draw[thick](1)-- (2)--(3)--(4)--(5);
        \end{tikzpicture} 
        &
        \begin{tikzpicture}[scale=0.5] 
         \node[fill=white] (1) at (0,-0.5) {};
         \node[fill=white] (2) at (1,-0.5) {};
          \node[fill=white] (3) at (2,-0.5) {};
         \node[fill=white] (4) at (3,-0.5) {};
         \node[fill=white] (5) at (4,-0.5) {};
          \tikzset{every loop/.style={min distance=4mm,in=120,out=60,looseness=20}}
           \foreach \i in {1,3} {
             \draw[thick,->] (\i) to[loop above] (\i);
         };
            \draw[thick](1)-- (2)--(3)--(4)--(5);
        \end{tikzpicture}
        &
        \begin{tikzpicture}[scale=0.5] 
         \node[fill=white] (1) at (0,-0.5) {};
         \node[fill=white] (2) at (1,-0.5) {};
          \node[fill=white] (3) at (2,-0.5) {};
         \node[fill=white] (4) at (3,-0.5) {};
         \node[fill=white] (5) at (4,-0.5) {};
          \tikzset{every loop/.style={min distance=4mm,in=120,out=60,looseness=20}}
           \foreach \i in {2} {
             \draw[thick,->] (\i) to[loop above] (\i);
         };
            \draw[thick](1)-- (2)--(3)--(4)--(5);
        \end{tikzpicture}\\
        &
        \begin{tikzpicture}[scale=0.5] 
         \node[fill=white] (1) at (0,-0.5) {};
         \node[fill=white] (2) at (1,-0.5) {};
          \node[fill=white] (3) at (2,-0.5) {};
                  \node[fill=white] (4) at (3,-0.5) {};
                  \node[fill=white] (5) at (4,-0.5) {};
          \tikzset{every loop/.style={min distance=4mm,in=120,out=60,looseness=20}}
           \foreach \i in {1,2,3,5} {
             \draw[thick,->] (\i) to[loop above] (\i);
         };
            \draw[thick](1)-- (2)--(3)--(4)--(5);
        \end{tikzpicture}
        \quad
          \begin{tikzpicture}[scale=0.5] 
         \node[fill=white] (1) at (0,-0.5) {};
         \node[fill=white] (2) at (1,-0.5) {};
          \node[fill=white] (3) at (2,-0.5) {};
         \node[fill=white] (4) at (3,-0.5) {};
         \node[fill=white] (5) at (4,-0.5) {};
          \tikzset{every loop/.style={min distance=4mm,in=120,out=60,looseness=20}}
           \foreach \i in {1,2,3,4,5} {
             \draw[thick,->] (\i) to[loop above] (\i);
         };
            \draw[thick](1)-- (2)--(3)--(4)--(5);
        \end{tikzpicture}
        &
        \begin{tikzpicture}[scale=0.5] 
         \node[fill=white] (1) at (0,-0.5) {};
         \node[fill=white] (2) at (1,-0.5) {};
          \node[fill=white] (3) at (2,-0.5) {};
         \node[fill=white] (4) at (3,-0.5) {};
         \node[fill=white] (5) at (4,-0.5) {};
          \tikzset{every loop/.style={min distance=4mm,in=120,out=60,looseness=20}}
           \foreach \i in {1,5} {
             \draw[thick,->] (\i) to[loop above] (\i);
         };
            \draw[thick](1)-- (2)--(3)--(4)--(5);
        \end{tikzpicture}
        &
        \begin{tikzpicture}[scale=0.5] 
         \node[fill=white] (1) at (0,-0.5) {};
         \node[fill=white] (2) at (1,-0.5) {};
          \node[fill=white] (3) at (2,-0.5) {};
         \node[fill=white] (4) at (3,-0.5) {};
         \node[fill=white] (5) at (4,-0.5) {};
          \tikzset{every loop/.style={min distance=4mm,in=120,out=60,looseness=20}}
           \foreach \i in {2,4} {
             \draw[thick,->] (\i) to[loop above] (\i);
         };
            \draw[thick](1)-- (2)--(3)--(4)--(5);
        \end{tikzpicture}\\
         &&
        \begin{tikzpicture}[scale=0.5] 
         \node[fill=white] (1) at (0,-0.5) {};
         \node[fill=white] (2) at (1,-0.5) {};
          \node[fill=white] (3) at (2,-0.5) {};
         \node[fill=white] (4) at (3,-0.5) {};
         \node[fill=white] (5) at (4,-0.5) {};
          \tikzset{every loop/.style={min distance=4mm,in=120,out=60,looseness=20}}
           \foreach \i in {2,3,4} {
             \draw[thick,->] (\i) to[loop above] (\i);
         };
            \draw[thick](1)-- (2)--(3)--(4)--(5);
        \end{tikzpicture}&\\
         &&
        \begin{tikzpicture}[scale=0.5] 
         \node[fill=white] (1) at (0,-0.5) {};
         \node[fill=white] (2) at (1,-0.5) {};
          \node[fill=white] (3) at (2,-0.5) {};
         \node[fill=white] (4) at (3,-0.5) {};
         \node[fill=white] (5) at (4,-0.5) {};
          \tikzset{every loop/.style={min distance=4mm,in=120,out=60,looseness=20}}
           \foreach \i in {1,2,4,5} {
             \draw[thick,->] (\i) to[loop above] (\i);
         };
            \draw[thick](1)-- (2)--(3)--(4)--(5);
        \end{tikzpicture}&
         \\
         \cline{2-4}
        &
        \begin{tikzpicture}[scale=0.5] 
         \node[fill=white] (1) at (0,-0.5) {};
         \node[fill=white] (2) at (1,-0.5) {};
          \node[fill=white] (3) at (2,-0.5) {};
         \node[fill=white] (4) at (3,-0.5) {};
         \node[fill=white] (5) at (4,-0.5) {};
          \tikzset{every loop/.style={min distance=4mm,in=120,out=60,looseness=20}}
           \foreach \i in {1,2,4} {
             \draw[thick,->] (\i) to[loop above] (\i);
         };
            \draw[thick](1)-- (2)--(3)--(4)--(5);
        \end{tikzpicture}&
        \begin{tikzpicture}[scale=0.5] 
         \node[fill=white] (1) at (0,-0.5) {};
         \node[fill=white] (2) at (1,-0.5) {};
          \node[fill=white] (3) at (2,-0.5) {};
         \node[fill=white] (4) at (3,-0.5) {};
         \node[fill=white] (5) at (4,-0.5) {};
          \tikzset{every loop/.style={min distance=4mm,in=120,out=60,looseness=20}}
           \foreach \i in {1,4} {
             \draw[thick,->] (\i) to[loop above] (\i);
         };
            \draw[thick](1)-- (2)--(3)--(4)--(5);
        \end{tikzpicture}&
         \\
         &&
        \begin{tikzpicture}[scale=0.5] 
         \node[fill=white] (1) at (0,-0.5) {};
         \node[fill=white] (2) at (1,-0.5) {};
          \node[fill=white] (3) at (2,-0.5) {};
         \node[fill=white] (4) at (3,-0.5) {};
         \node[fill=white] (5) at (4,-0.5) {};
          \tikzset{every loop/.style={min distance=4mm,in=120,out=60,looseness=20}}
           \foreach \i in {2,3} {
             \draw[thick,->] (\i) to[loop above] (\i);
         };
            \draw[thick](1)-- (2)--(3)--(4)--(5);
        \end{tikzpicture}&
         \\
&&
        \begin{tikzpicture}[scale=0.5] 
         \node[fill=white] (1) at (0,-0.5) {};
         \node[fill=white] (2) at (1,-0.5) {};
          \node[fill=white] (3) at (2,-0.5) {};
         \node[fill=white] (4) at (3,-0.5) {};
         \node[fill=white] (5) at (4,-0.5) {};
          \tikzset{every loop/.style={min distance=4mm,in=120,out=60,looseness=20}}
           \foreach \i in {2,1,5} {
             \draw[thick,->] (\i) to[loop above] (\i);
         };
            \draw[thick](1)-- (2)--(3)--(4)--(5);
        \end{tikzpicture}&
         \\
&&
        \begin{tikzpicture}[scale=0.5] 
         \node[fill=white] (1) at (0,-0.5) {};
         \node[fill=white] (2) at (1,-0.5) {};
          \node[fill=white] (3) at (2,-0.5) {};
         \node[fill=white] (4) at (3,-0.5) {};
         \node[fill=white] (5) at (4,-0.5) {};
          \tikzset{every loop/.style={min distance=4mm,in=120,out=60,looseness=20}}
           \foreach \i in {1,3,4} {
             \draw[thick,->] (\i) to[loop above] (\i);
         };
            \draw[thick](1)-- (2)--(3)--(4)--(5);
        \end{tikzpicture}&
         \\
&&
        \begin{tikzpicture}[scale=0.5] 
         \node[fill=white] (1) at (0,-0.5) {};
         \node[fill=white] (2) at (1,-0.5) {};
          \node[fill=white] (3) at (2,-0.5) {};
         \node[fill=white] (4) at (3,-0.5) {};
         \node[fill=white] (5) at (4,-0.5) {};
          \tikzset{every loop/.style={min distance=4mm,in=120,out=60,looseness=20}}
           \foreach \i in {1,3,5} {
             \draw[thick,->] (\i) to[loop above] (\i);
         };
            \draw[thick](1)-- (2)--(3)--(4)--(5);
        \end{tikzpicture}&
         \\
         \hline         
    \end{tabular}
    \caption{Answer to Problem~\ref{problem-6.5} for $n\leq 5$, where every undirected edge represents a pair of directed arcs in opposite directions. Results in Section~\ref{sec:double-paths} characterize all digraphs with $n\leq 4$ and the ones above the bottom line for $n=5$, whereas those below the line represent cases resolved ad hoc in Example~\ref{ex:n<=5}.}  \label{tab:small-paths}
\end{table}

\subsection{Minimum number of nonzero entries in the pattern that requires the nSSP}

An \emph{$n\times n$ sign pattern} $P$ is an $n\times n$ matrix whose entries are $+$, $-$, or $0$. A matrix $A\in \bR^{n\times n}$ \emph{has the sign pattern $P$} if each entry of $A$ has the sign of the corresponding entry of $P$. By $\calQ(P)\subseteq \bR^{n\times n}$ we denote the set of all matrices of the same dimensions as $P$ that have the sign pattern $P$. If every matrix $A \in \calQ(P)$ has the nSSP, then we say that $P$ \emph{requires the nSSP}.

In~\cite[Problem~6.1]{Arav-24-nSSP} the authors ask whether for $n\geq 3$, the minimum number of nonzero entries in an $n\times n$ sign pattern that requires the STP, is $2n-1$.  In the next example, we show that the answer to this question is negative as well.

\begin{example}\label{ex:2n-2-arcs-notnSSP}
Consider the $3\times 3$ sign pattern $$P=\begin{pmatrix}
    +&0&-\\
    0&-&0\\
    +&0&0
\end{pmatrix}$$
with four nonzero entries. Let $A=\begin{pmatrix}
    a_{11}&0&a_{13}\\
    0&a_{22}&0\\
    a_{31}&0&0
\end{pmatrix} \in \calQ(P)$ be an arbitrary matrix with pattern $P$ and suppose that $X=\begin{pmatrix}
    0&x_{12}&0\\
    x_{21}&0&x_{23}\\
    0&x_{32}&x_{33}
\end{pmatrix}$ satisfies $[A,X\trans]=O$. Then ${\bf x}=(x_{12},x_{21},x_{23},x_{32},x_{33})\trans$ has to be a solution of the homogeneous linear system 
$$\begin{pmatrix}
-a_{13} & 0 &0 & a_{22}&0\\
0&a_{31} & -a_{22}&0&0\\
0&a_{11}-a_{22} & a_{13}&0&0\\
a_{22}-a_{11} & 0 &0 & -a_{31}&0\\
0&0 &0&0&a_{13}\\
0&0 &0&0&a_{31}\\
\end{pmatrix}{\bf x}={\bf 0}.$$
Note that the determinant of the $5\times 5$ submatrix induced by the first five rows is equal to $a_{13}(a_{13}a_{31}+a_{22}(a_{11}-a_{22}))^2
$.
Since $a_{13}\ne 0$ and $a_{13}a_{31}+a_{22}(a_{11}-a_{22})<0$, it follows that the matrix of the system is of full column rank and hence ${\bf x}={\bf 0}$. This implies $X=O$ and so the pattern $P$ with $2n-2$ nonzero entries requires the nSSP.

However, the zero--nonzero $3\times 3$ pattern with the same positions of zeros as the pattern $P$ allows, but does not require the nSSP since for matrices $$A=\begin{pmatrix}
 1 & 0 & 1 \\
 0 & 2 & 0 \\
 1 & 0 & 0 \\
\end{pmatrix}, \; B=\begin{pmatrix}
    1 & 0 & 2 \\
 0 & -1 & 0 \\
 1 & 0 & 0 
\end{pmatrix} \text{ and } Y=\begin{pmatrix}
0 & 0 & 0 \\
 1 & 0 & -1 \\
 0 & 0 & 0 \\
\end{pmatrix}$$
it is easy to verify that $A$ has the nSSP and that  $B\circ Y=[B,Y\trans]=O$.
\end{example}

This result motivates the investigation of~\cite[Problem~6.2]{Arav-24-nSSP}, which we equivalently reformulate in the following problem. 

\begin{problem}\label{problem:minimum-number-arcs-not-require}
Is the minimum number of arcs in a directed graph $G$ on $n$ vertices that requires the nSSP equal to $2n-1$?  
\end{problem}

In Section~\ref{sec:Problem 6.1} we investigate sparse  digraphs. In Theorems~\ref{thm:does not require} and~\ref{thm:does not allow} we find some sufficient conditions for a digraph not to require or not to allow the nSSP. 
These results show that for several families of graphs the answer to Problem~\ref{problem:minimum-number-arcs-not-require} is affirmative, and that one can limit the investigation of the Problem~\ref{problem:minimum-number-arcs-not-require} to the strongly connected digraphs with at least one double arc. This result also enables us to verify that all directed graphs on $n\leq 4$ vertices with at most $2n-2$ arcs do not require the nSSP, see Remark~\ref{rem:small-graphs}.

\subsection{Notation }\label{sec:notation}

In this paper we use the following notation for a directed graph  $G=(V(G),E(G))$, which we call a digraph for short. Every arc $a\in E(G)$ is of the form $a=(u,v)$, where $u,v\in V(G)$. In the case $u=v$, we call $(v,v)\in E(G)$ \emph{a loop} and we denote by  $\loops(G)$ the set of all loops in $G$. If $\loops(G)=\emptyset$, we call the digraph $G$ \emph{loopless digraph}. If a digraph contains arcs $(u,v)$ and $(v,u)$, we call $(u,v)$ a double arc and sometimes replace $(u,v)$ and $(v,u)$ by $\{u,v\}$.

A path from $u$ to $v$ in a directed graph $G$ is a sequence $v_0v_1...v_k$ of vertices in $V(G)$ such that $(v_i,v_{i+1}) \in E(G)$ for all $0\leq i\leq k-1$, while $v_0=u$ and $v_k=v$.
A directed graph $G$ is \emph{strongly connected} if for every pair of vertices $(u, v) \in V(G) \times V(G)$ there is a path from $u$ to $v$. A directed graph
$G$ is \emph{weakly connected}  if the undirected graph obtained from 
$G$ by replacing each arc with an undirected edge (also called the \emph{underlying simple graph}) is connected.
\emph{A strongly connected component} of $G$ is a maximal strongly connected subgraph. We allow a vertex (with or without a loop) to be a strongly connected component. \emph{A strong bridge} is an arc $e\in E(G)$ whose removal results in a graph $G\setminus e$, which has more strongly connected components than $G$.

We denote by 
$$N_G^{+}[v]:=\{u\in V(G)\colon (v,u)\in E(G)\} \text{ and } N_G^{-}[v]:=\{w\in V(G)\colon (w,v)\in E(G)\}$$ the sets of all out-neighbours and in-neighbours of $v\in V(G)$, and $N_G[v]:=N_G^{+}[v]\cup N_G^{-}[v]$.  Note that $v\in\loops(G)$ if and only if $v\in N_G^+[v]\cap N_G^-[v]$.

The complement of $G$ is denoted by $G^c$. If ${\mathcal V} \subseteq V(G)$, we define $G[{\mathcal V}]$ to be the induced subgraph on vertices ${\mathcal V}$. 
If $G$ and $H$ are two directed graphs with the same set of vertices and $E(G)= E(H) \cup \{e_1,\ldots, e_k\}$, we abbreviate $G=H+\{e_1,\ldots, e_k\}$. If $k=1$, we also write $G=H+e_1$, and if all the arcs are double, we write
$H+\{\{u,v\} \colon  u\in {\mathcal U},v\in {\mathcal V} \}:=H+\{(v,u),(u,v)\colon  u\in {\mathcal U},v\in {\mathcal V} \}$
for some ${\mathcal U}, {\mathcal V} \subseteq V(G)$.

For any $n\in\bN$ let $[n]=\{1,2,\ldots,n\}$. 
We use $\bf 0$ to denote the zero vector and by ${\bf e}_i$ the vector with the only nonzero entry equal to 1 in the $i$th position. Let $O$ to denote the zero matrix and $I$ the identity matrix. To emphasise their sizes, we will sometimes use indices as subscripts. Moreover, for two square matrices of the same size, we denote by $A \circ B$ their Haddamard product and by $[A,B]=AB-BA$ their commutator.

The \emph{adjacency matrix of a digraph $G$} with $V(G)=[n]$ is the matrix $A(G) = (a_{ij})\in\{0,1\}^{n\times n}$
where $a_{ij} = 1$
if and only if $(i, j) \in E(G)$.
By ${\mathcal M}(G)$ we denote the set of all real $n\times n$ matrices that have the same zero-nonzero pattern as the adjacency matrix $A(G)$, i.e. $(i,j)$-entry of $B\in {\mathcal M}(G)$ is nonzero if and only if $(i,j)\in E(G)$. By $\overline{\mathcal M}(G)$ we denote the set of all 
real $n\times n$ matrices whose $(i,j)$-entry is zero if $(i,j)\notin E(G)$. Note that for $(i,j)\in E(G)$ matrix $M\in \overline{\mathcal M}(G)$ can have $M_{i,j}=0$ or $M_{i,j}\ne 0$. 
Moreover, if $A\in {\mathcal M}(G)$ and $A\circ X=O$, then
$X\in \overline{\mathcal M}(G^c)$.

   \section{On the minimum number of arcs in a directed graph that requires the nSSP}\label{sec:Problem 6.1}

In this secton we investigate directed graphs $G$ with $n$ vertices and with small number of arcs. We present several examples that give positive answer to Problem~\ref{problem:minimum-number-arcs-not-require}.  First we present a subfamily of Hessenberg patterns, for which Saha et.~al.~in~\cite[Theorem 3.11]{2025-Saha-Tilis-VanderMeulen-VanTuyl-nSMP} showed that they require the nSMP, i.e. for each matrix $A\in{\mathcal M}(G)$, the only matrix satisfying $A\circ X=[A,X\trans]=O$ and ${\rm trace}(X\trans A^{k-1})=0$ for all $k\in [n]$, is $X=O$.
 Their proof establishes also the next result, which shows the difference between the nSSP and the nSMP, and at the same time gives an example of a digraph with  $2n-1$ arcs that requires the nSSP.

\begin{lemma}\label{lem:Hessenberg}
Let $G$ be a digraph with $V(G)=[n]$ and
     $$\{(i,i+1)\colon i\in [n-1]\}\subseteq E(G)\subseteq\{(i,j)\colon i\geq j-1\}$$
and assume $G$ has at least one directed $k$-cycle for each $k,\ 2\leq k \leq n$.

If $\loops(G)=\emptyset$, then $G$ does not allow the nSSP, but it requires the nSMP, and otherwise if $\loops(G)\ne \emptyset$, then $G$ requires the nSSP.
\end{lemma}

\begin{proof}
The first part of the statement follows by Lemma~\ref{lem:noloops-all-loops-nssp} and~\cite[Theorem 3.11]{2025-Saha-Tilis-VanderMeulen-VanTuyl-nSMP}. If $G$ has at least one loop, we follow the proof of the mentioned theorem to observe that if $A\in{\mathcal M}(G)$ has at least one 
nonzero entry on each subdiagonal, then any matrix $X \in \overline{\mathcal{M}}(G^c)$ satisfying $[A,X\trans]=0$ is of the form $X\trans = c_0 I$ for some $c_0 \in \mathbb{R}$. 
Since $A$ has at least one nonzero entry on the diagonal as well, it follows that $c_0 = 0$, and hence $X = O$.  Consequently, $A$ has the nSSP and $G$ requires the nSSP. 
\end{proof}

By considering the structural decomposition of a digraph, one can extend the nSSP property from subgraphs, as stated in the following theorem from~\cite{Arav-24-nSSP}.

\begin{lemma}\label{lem:arav-reducible}\cite[Theorem~3.10]{Arav-24-nSSP}
    Suppose that a digraph $G$ is not strongly (weakly, respectively) connected and has strongly (weakly, respectively) connected components $G_1,\ldots,G_k$, $k\geq 2$. A matrix $A\in {\mathcal M}(G)$ has the nSSP if and only if $A[G_i]$ has the nSSP and $\spec(A[G_i]) \cap \spec(A[G_j]) =\emptyset$ for all distinct $i,j\in [k]$.
\end{lemma}

The next Corollary follows from  Lemma~\ref{lem:arav-reducible} and shows that graphs that do not allow the nSSP are abundant. We state it in our setting of directed graphs and give an alternative proof.

\begin{corollary} \label{cor:reducible-not-allow-nSSP}
    Let $G$ be a directed graph, such that $V(G)=V(G_1)\cup V(H)$ and
       either $N_G^+[u]\subseteq V(G_1)$ for all $u\in V(G_1)$ or $N_G^-[u]\subseteq V(G_1)$ for all $u\in V(G_1)$.
    
      If $H$ does not allow the nSSP (does not require the nSSP, respectively),
then $G$ does not allow the nSSP (does not require the nSSP, respectively).
\end{corollary}

\begin{proof}
    Let $k=|V(H)|$ and let us first assume that $N_G^-[u]\subseteq V(G_1)$ for all $u\in V(G_1)$. Then every matrix $A\in {\mathcal M}(G)$ has a form
    $$A=\begin{pmatrix}
        A_1 & B\\
        O & C
    \end{pmatrix},$$
    where $C\in{\mathcal M}(H)$. Since $H$ does not allow the nSSP, it follows that there exists $Y_0$, such that $C\circ Y_0=[C,Y_0\trans]=O$. Let $Z=(z_{ij})\in \bR^{k\times (n-k)}$ be a matrix with $k(n-k)$ parameters and let  $\bfz=\vecop(Z\trans)\in \bR^{k(n-k)}$ be a vector obtained by stacking the columns of $Z\trans$ one above the other.
    
    We claim there exists a nonzero matrix $X:=\begin{pmatrix}
        O& O\\
        Z & \alpha Y_0
    \end{pmatrix}$ that solves  $A\circ X=[A,X\trans]=O$. It is straightforward to check that if $B=O$, then the two equations are satisfied for $X$ with $Z=O$. 
    Otherwise, note that the first equality $A\circ X=O$ holds for all $Z$ and $\alpha$. 
    Observe that
    $[A,X\trans]=\begin{pmatrix}
        O & A_1 Z\trans+\alpha B Y_0\trans-Z\trans C\\
        O & O
    \end{pmatrix}$, and so we have to prove that \begin{equation}\label{eq:top-left-corner} 
       A_1 Z\trans I_k-I_{n-k}Z\trans C+B(\alpha I_k) Y_0\trans=O
    \end{equation}
    has nontrivial solution in variables $\alpha\in\bR$ and $\bfz\in \bR^{k(n-k)}$. By vectorization,  equation~\eqref{eq:top-left-corner}
    is equivalent to 
    $(I_k\otimes A_1-C\trans \otimes I_{n-k})\bfz  +(Y_0\otimes B) (\alpha \vecop(I_k))=0$.
    Note that all entries of vector $\alpha \vecop(I_k)$ are equal to either zero or $\alpha$, and so $(Y_0\otimes B) (\alpha \vecop(I_k))=\alpha \bfb$, where $\bfb$ is the sum of those columns of matrix $Y_0\otimes B$, that are in the position of nonzero entries in $\vecop(I_k)$. This implies solving~\eqref{eq:top-left-corner} is equivalent to solving a homogeneous linear system
    $$\begin{pmatrix}
        I_k\otimes A_1-C\trans\otimes I_{n-k} && \bfb
    \end{pmatrix}\begin{pmatrix}
        \bfz\\
        \alpha
    \end{pmatrix}={\bf 0}.$$
   
    Since the matrix of this system has $k(n-k)$ rows and $k(n-k)+1$ columns, it follows that it has a nontrivial solution, and so $A$ does not have the nSSP.

    We prove the case $N_G^+[u]\subseteq V(G_1)$ similarly.
\end{proof}

In the next theorem we present several families of graphs that do not require the nSSP and can be used as building blocks in Corollary~\ref{cor:reducible-not-allow-nSSP}.

\begin{theorem}\label{thm:does not require}
    Suppose a digraph $G$ has one of the following properties:
    \begin{enumerate}\renewcommand{\theenumi}{\Alph{enumi}}
        \item\label{1:not-require-reducible}  $G$ is not strongly connected.
        \item\label{2:not-require-no-double-arcs}  $G$ has no double arcs.
      
        \item\label{3:not-require-bipartite} 
        $V(G)=[2m]$ with $\loops(G)=[m]$ for some $m\geq 2$, and
        \begin{itemize}
            \item $(i+m,i)$ for all $i\in [m]$, 
            \item $(i,i+m+1) \!\!\!\mod 2m$ for all $i\in[m]$, and
            \item a subset of $\{(i,j+m)\colon i,j\in [m]\}$, such that if $(i,j+m)\in E(G)$ for some $i\ne j\in[m]$ then $(j,i+m)\notin E(G)$.
        \end{itemize}
    \end{enumerate}
    Then $G$ does not require the nSSP.
\end{theorem}

\begin{proof}

\begin{enumerate}\renewcommand{\theenumi}{\Alph{enumi}}
\item Suppose $G$ has strongly connected components $G_1,\ldots,G_k$, $k\geq 2$. Let $A_1$ and $A_2$ be adjacency matrices of $G_1$ and $G_2$, respectively. Since $A_1$ and $A_2$ are entry-wise nonnegative and irreducible, they have a real positive (Perron) eigenvalue $\lambda_1$ and $\lambda_2$, respectively. 
If $A\in {\mathcal M}(G)$ is a matrix such that $A[G_1]=\lambda_2 A_1$ and $A[G_2]=\lambda_1 A_2$, then $\lambda_1\lambda_2\in \spec(A[G_1]) \cap \spec(A[G_2])$ and so by Lemma~\ref{lem:arav-reducible} matrix $A$ does not have the nSSP. 
    \item Let $A\in{\mathcal{M}}(G)$ be the adjacency matrix of $G$. Let us define $X:=I-A\trans\in\{0,1\}^{n\times n}$. Note that since $A$ has no double arcs $A\circ X=A\circ(I-A\trans)=O$, 
    and $[A,X\trans]=[A,I-A]=O$, therefore $A$ does not have the nSSP. 
    \item Let $A$ be the adjacency matrix of $G$ and note that it is of the form $$A=\begin{pmatrix}
        I & B\\
        I & 0
    \end{pmatrix},$$
    where $B$ is the adjacency matrix of a graph $H$ with $m$ vertices and without double arcs. By~(\ref{2:not-require-no-double-arcs}) there exists a nonzero matrix $Y$, such that $B\circ Y=[B,Y\trans]=O$. Let $X=\begin{pmatrix}
        O&-Y\\
        -B\trans Y&Y
    \end{pmatrix}$.
    It is straightforward to check that $[A,X\trans]=O$. Moreover, since $B\circ Y=O$, it follows that $(B\trans Y)\circ I=O$ and therefore $A\circ X=O$ as well. This implies that $A$ does not have the nSSP.
 \end{enumerate}
 In each of the cases we constructed a matrix $A$ that does not have the nSSP and so $G$ does not require the nSSP.
\end{proof}

Note that every strongly connected graph is weakly connected and so if a digraph $G$ is not weakly connected, then by Theorem~\ref{thm:does not require}(\ref{1:not-require-reducible}) implies it does not require the nSSP. 

\begin{example}
Note that a digraph which is not strongly connected, may still allow the nSSP. 
 For example, for a digraph $G$ on Figure~\ref{fig:directed-graph-reducible-allow-nSSP}, it is straightforward to check that the matrix
 $A=\begin{pmatrix}
 -1 & 1 & 0 \\
 1 & 0 & 1 \\
 0 & 0 & 1 
 \end{pmatrix}\in{\mathcal M}(G)
 $ has the nSSP, so it allows the nSSP but does not require the nSSP by Theorem~\ref{thm:does not require}. 
On the other hand, let $G'$ denote the same digraph without a loop on vertex $3$, see Figure~\ref{fig:directed-graph-reducible-allow-nSSP}. Let $G_1=G'[\{1,2\}]$ and $H=G'[\{3\}]$. Since $H$ does not allow the nSSP, it follows by Corollary~\ref{cor:reducible-not-allow-nSSP} that $G'$ does not allow it either.

 \begin{figure}[h]
\begin{tikzpicture}
\foreach \i in {1,...,3} {
    \node[fill=white] (\i) at (\i,0) {};
    \node[rectangle, draw=none] at (\i,-0.7) {$\i$};
};

\tikzset{every loop/.style={min distance=4mm,in=120,out=60,looseness=30}}
\foreach \i in {1,3} {
     \draw[thick,->] (\i) to[loop above] (\i);
};

\draw[->, thick] (1) to[bend left] (2);
\draw[->, thick] (2) to[bend left] (1);
\draw[->, thick] (2) -- (3);
\node[rectangle, draw=none] at (2,-1.5) {$G$};
\end{tikzpicture}
\qquad\qquad
\begin{tikzpicture}
\foreach \i in {1,...,3} {
    \node[fill=white] (\i) at (\i,0) {};
    \node[rectangle, draw=none] at (\i,-0.7) {$\i$};
};

\tikzset{every loop/.style={min distance=4mm,in=120,out=60,looseness=30}}
\foreach \i in {1} {
     \draw[thick,->] (\i) to[loop above] (\i);
};

\draw[->, thick] (1) to[bend left] (2);
\draw[->, thick] (2) to[bend left] (1);
\draw[->, thick] (2) -- (3);
\node[rectangle, draw=none] at (2,-1.5) {$G'$};
\end{tikzpicture}
\caption{An example of a weakly but not strongly connected digraph $G$ that allows but does not require the nSSP, and $G'$ that does not allow the nSSP.}\label{fig:directed-graph-reducible-allow-nSSP}
\end{figure}
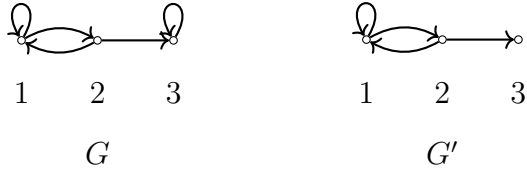

\end{example}

\begin{example}
 Theorem~\ref{thm:does not require}(\ref{3:not-require-bipartite}) provides a family of strongly connected digraphs with possible double arcs that do not require the nSSP. For example, for $m=4$, a family of supergraphs of a directed cycle is shown of Figure~\ref{fig:Cn-not-require}, where the black arcs are forced by the first two items and the blue arcs are optional by the third item.

   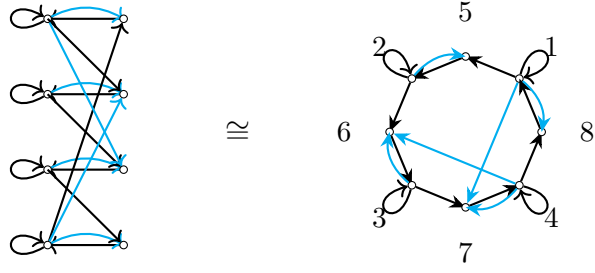
\begin{figure}[htb]
\begin{tikzpicture}
\tikzset{every loop/.style={min distance=4mm,in=150,out=210,looseness=30}}

\foreach \i in {5,...,8} {
    \draw[->, thick] (1,\i-4) -- (0,\i-4);
    \draw[->, thick,cyan] (0,\i-4) to[bend left] (1,\i-4);
     \node[fill=white] (\i) at (1,\i-4) {};
};
\foreach \i in {1,...,4} {
    \node[fill=white] (\i) at (0,\i) {};
    \draw[thick,->] (\i) to[loop above] (\i);
};
\foreach \i in {5,6,7} {   
\draw[->, thick]  (0,\i-3)--(\i);
}
\draw[->, thick] (1) -- (8);
\draw[->, thick,cyan] (4)-- (6);
\draw[->, thick,cyan] (1)-- (7);

\begin{scope}[shift={(5.5,2.5)}]
\foreach \i in {1,...,8} {
 \node[fill=white] (\i) at (\i*45:1) {};
  \draw[-{Stealth[length=6pt,width=5pt]}, thick]  (\i) -- (45+\i*45:1);
};
\foreach \i/\lab in {1/1, 2/5, 3/2, 4/6, 5/3, 6/7, 7/4, 8/8} {
    \pgfmathsetmacro{\angle}{\i*45}
    \node[draw=none] at ($( \i ) + (\angle:0.6)$) {{\small \lab}};
}
\foreach \i in {1,3,5,7} {
    \pgfmathsetmacro{\angle}{\i*45}
    \draw[->, thick]
        (\i) to[out=\angle+30, in=\angle-30, looseness=30] (\i);
    \draw[-{Stealth[length=6pt,width=5pt]},thick,cyan] (\i) to[bend left] (-45+\i*45:1);
};
\node[draw=none] at (-3,0) {$\cong$};
\draw[-{Stealth[length=6pt,width=5pt]},thick,cyan] (7) -- (4);
\draw[-{Stealth[length=6pt,width=5pt]},thick,cyan] (1) -- (6);
\end{scope}
\end{tikzpicture}
\caption{An example of a family of supergraphs of a directed cycle that does not require the nSSP; the black arcs have to be in a graph, but the blue arcs are optional.}\label{fig:Cn-not-require}
\end{figure}
\end{example}

The following result provides some sufficient conditions under which a digraph does not allow the nSSP, thereby offering a partial answer to Problem~\ref{problem:minimum-number-arcs-not-require}.  Note that if a digraph does not allow the nSSP, then also every sign pattern that corresponds to this digraph does not allow the nSSP, so the next theorem also presents a family of sign patterns for which~\cite[Problem 6.3]{Arav-24-nSSP} is resolved.

\begin{theorem}\label{thm:does not allow}
Suppose $G$ is a directed graph on $n$ vertices with one of the following properties:

\begin{enumerate}\renewcommand{\theenumi}{\Alph{enumi}}
 \item $|G|\geq 2$, $G$ has no double arcs and $|\loops(G)|=1$.
\item\label{numberarcs} $|E(G)|\leq 2n-2$, and its  underlying simple graph is a tree.
\end{enumerate}
Then $G$ does not allow the nSSP. 
\end{theorem} 

\begin{proof} 
\begin{enumerate}\renewcommand{\theenumi}{\Alph{enumi}}
\item Let
$A=(a_{ij})\in{\mathcal{M}}(G)$  and without loss of generality assume $a_{11}\ne 0$.
Let $X=(x_{ij})$ defined by
$X:=a_{11}I-A\trans\in\bR^{n\times n}$.
Note the $x_{11}=0$ and $x_{ii}=a_{11}\ne 0$ for $i\geq 2$.
Since there are no double arcs in $G$, it follows that $A\circ X=O$, and matrix $X\trans=a_{11}I-A$  commutes with $A$. 

\item If the underlying simple graph of $G$ is a tree, then $G$ is weakly connected.
We will prove this part using induction on the number of vertices. In the case $n=1$, $G$ is a vertex without a loop and does not allow the nSSP by Lemma~\ref{lem:noloops-all-loops-nssp}. 

If $n\geq 2$ and $\mathcal{L}(G)=\emptyset$, then the statement follows by Lemma~\ref{lem:noloops-all-loops-nssp}, so suppose now $n\geq 2$ and $\mathcal{L}(G)\neq \emptyset$. 

Let us assume now that the theorem is true for all digraphs with less than $n$ vertices satisfying the conditions, and let $G$ be a digraph on $n$ vertices  such that $|E(G)|\leq 2n-2$, and its underlying simple graph is a tree. 
Because $|E(G)|\leq 2n-2$ there is at least one arc $e=(u,v)\in E(G)$, such that $(v,u) \notin E(G)$. 
Note that $e$ is a strong bridge in $G$ since its underlying simple graph contains no cycles, and let us denote by $G'$ and $G''$ the components of the digraph $G\setminus e$ after deletion of arc $e$.  Let $|V(G')|=m$ and $|V(G'')|=n-m$. If $|E(G')|\geq 2m-1$ and $|E(G'')|\geq 2(n-m)-1$, then $|E(G)|=|E(G')|+|E(G'')|+1\geq 2n-1$, which contradicts our assumption on $|E(G)|\leq 2n-2$. Therefore assume without loss of generality that $|E(G')|\leq 2m-2$ and by inductive hypothesis $G'$ does not allow the nSSP. By Corollary~\ref{cor:reducible-not-allow-nSSP} $G$ does not allow the nSSP.\qedhere
\end{enumerate}

\end{proof}

Theorems~\ref{thm:does not require} and~\ref{thm:does not allow}, together with Corollary~\ref{cor:reducible-not-allow-nSSP} provide examples of graphs that give confirmative answer to Problem~\ref{problem:minimum-number-arcs-not-require}.    

 \begin{remark}\label{rem:small-graphs}
   In order to verify that the answer to Problem~\ref{problem:minimum-number-arcs-not-require} is affirmative for a chosen $n$, based on results in Theorems~\ref{thm:does not require} and~\ref{thm:does not allow} and Corollary~\ref{cor:reducible-not-allow-nSSP} it is enough to consider strongly connected directed graphs $G$ that have the following properties:
  
   \begin{enumerate}\renewcommand{\theenumi}{\Alph{enumi}}
       \item\label{item:arcs} $|E(G)|\leq 2n-2$,
       \item $G$ is strongly connected and \item\label{item:loop-edge} $G$ has a loop and a double arc.
   \end{enumerate}

  For $n\leq 3$ there are no connected directed graphs with properties (\ref{item:arcs})-(\ref{item:loop-edge}), and hence all of them give the positive answer to Problem~\ref{problem:minimum-number-arcs-not-allow}. 
  
  For $n=4$ there are only eight non-isomorphic directed graphs with properties (\ref{item:arcs})-(\ref{item:loop-edge}), shown on Figure~\ref{fig:4-vertices}. It is straightforward to show that all eight adjacency matrices corresponding to these digraphs do not have the nSSP. Therefore,  Problem~\ref{problem:minimum-number-arcs-not-allow} has a positive answer for $n=4$ as well.
  
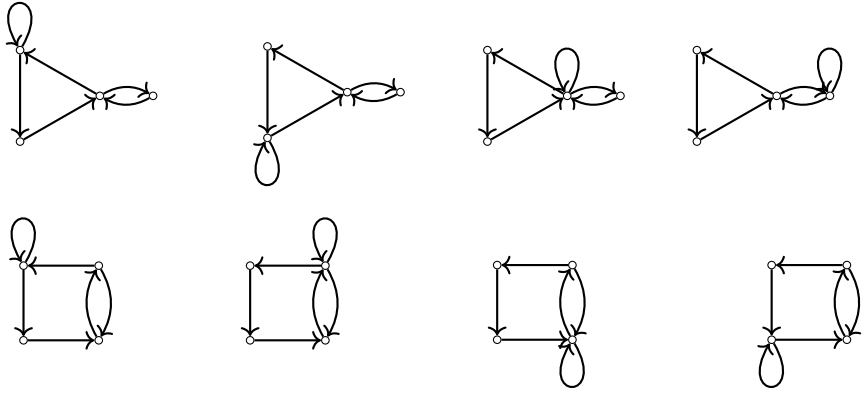
\begin{figure}[h]
\begin{tikzpicture}[scale=0.7]
\foreach \i in {1,...,3} {
    \node[fill=white] (\i) at (\i*120:1) {};
    
};
 \node[fill=white] (4) at (2,0) {};
\phantom{\node[fill=white] at (0,-1.9) {};}

\tikzset{every loop/.style={min distance=4mm,in=120,out=60,looseness=40}}
\draw[thick,->] (1) to[loop above] (1);

\draw[->, thick] (1) -- (2);
\draw[->, thick] (2) -- (3);
\draw[->, thick] (3) -- (1);
\draw[->, thick] (3) to[bend left] (4);
\draw[->, thick] (4) to[bend left] (3);

\end{tikzpicture}
\qquad
\begin{tikzpicture}[scale=0.7]
\foreach \i in {1,...,3} {
    \node[fill=white] (\i) at (\i*120:1) {};
    
};
 \node[fill=white] (4) at (2,0) {};

\tikzset{every loop/.style={min distance=4mm,in=240,out=300,looseness=40}}
\draw[thick,->] (2) to[loop above] (2);

\draw[->, thick] (1) -- (2);
\draw[->, thick] (2) -- (3);
\draw[->, thick] (3) -- (1);
\draw[->, thick] (3) to[bend left] (4);
\draw[->, thick] (4) to[bend left] (3);

\end{tikzpicture}
\qquad
\begin{tikzpicture}[scale=0.7]
\foreach \i in {1,...,3} {
    \node[fill=white] (\i) at (\i*120:1) {};
    
};
 \node[fill=white] (4) at (2,0) {};
\phantom{\node[fill=white] at (0,-1.9) {};}
\tikzset{every loop/.style={min distance=4mm,in=120,out=60,looseness=40}}
\draw[thick,->] (3) to[loop above] (3);

\draw[->, thick] (1) -- (2);
\draw[->, thick] (2) -- (3);
\draw[->, thick] (3) -- (1);
\draw[->, thick] (3) to[bend left] (4);
\draw[->, thick] (4) to[bend left] (3);

\end{tikzpicture}\qquad
\begin{tikzpicture}[scale=0.7]
\foreach \i in {1,...,3} {
    \node[fill=white] (\i) at (\i*120:1) {};
    
};
 \node[fill=white] (4) at (2,0) {};
\phantom{\node[fill=white] at (0,-1.9) {};}
\tikzset{every loop/.style={min distance=4mm,in=120,out=60,looseness=40}}
\draw[thick,->] (4) to[loop above] (4);

\draw[->, thick] (1) -- (2);
\draw[->, thick] (2) -- (3);
\draw[->, thick] (3) -- (1);
\draw[->, thick] (3) to[bend left] (4);
\draw[->, thick] (4) to[bend left] (3);

\end{tikzpicture}

\begin{tikzpicture}[scale=0.7]
\foreach \i in {1,...,4} {
    \node[fill=white] (\i) at (\i*90+45:1) {};
    
};
\phantom{\node[fill=white] at (0,-1.8) {};}
\tikzset{every loop/.style={min distance=4mm,in=120,out=60,looseness=40}}
\draw[thick,->] (1) to[loop above] (1);

\draw[->, thick] (1) -- (2);
\draw[->, thick] (2) -- (3);
\draw[->, thick] (3) to[bend left] (4);
\draw[->, thick] (4) to[bend left] (3);
\draw[->, thick] (4) -- (1);
\end{tikzpicture}\qquad\qquad
\begin{tikzpicture}[scale=0.7]
\foreach \i in {1,...,4} {
    \node[fill=white] (\i) at (\i*90+45:1) {};
    
};

\tikzset{every loop/.style={min distance=4mm,in=120,out=60,looseness=40}}
\draw[thick,->] (4) to[loop above] (4);

 \phantom{\node[fill=white] at (0,-1.8) {};}
\draw[->, thick] (1) -- (2);
\draw[->, thick] (2) -- (3);
\draw[->, thick] (3) to[bend left] (4);
\draw[->, thick] (4) to[bend left] (3);
\draw[->, thick] (4) -- (1);
\end{tikzpicture}\qquad\qquad
\begin{tikzpicture}[scale=0.7]
\foreach \i in {1,...,4} {
    \node[fill=white] (\i) at (\i*90+45:1) {};
    
};

\tikzset{every loop/.style={min distance=4mm,in=240,out=300,looseness=40}}
\draw[thick,->] (3) to[loop above] (3);

\draw[->, thick] (1) -- (2);
\draw[->, thick] (2) -- (3);
\draw[->, thick] (3) to[bend left] (4);
\draw[->, thick] (4) to[bend left] (3);
\draw[->, thick] (4) -- (1);
\end{tikzpicture}\qquad\qquad
\begin{tikzpicture}[scale=0.7]
\foreach \i in {1,...,4} {
    \node[fill=white] (\i) at (\i*90+45:1) {};
    
};

\tikzset{every loop/.style={min distance=4mm,in=240,out=300,looseness=40}}
\draw[thick,->] (2) to[loop above] (2);

\draw[->, thick] (1) -- (2);
\draw[->, thick] (2) -- (3);
\draw[->, thick] (3) to[bend left] (4);
\draw[->, thick] (4) to[bend left] (3);
\draw[->, thick] (4) -- (1);
\end{tikzpicture}
\caption{All graphs on four vertices satisfying (\ref{item:arcs})-(\ref{item:loop-edge}).}\label{fig:4-vertices}
\end{figure}

Using a Wolfram Mathematica code we verified that Problem~\ref{problem:minimum-number-arcs-not-require} has a positive answer for $n=5$ as well, and so all connected directed graphs on at most $n\leq 5$ vertices and at most $2n-2$ arcs do not require the nSSP.
 \end{remark}

These results together with the Example \ref{ex:2n-2-arcs-notnSSP} motivate the following two questions.

\begin{problem}\label{problem:minimum-number-arcs-not-allow}
Is the minimum number of arcs in a strongly connected  directed graph $G$ on $n$ vertices that requires the nSSP equal to $2n-1$? Is the minimum number of arcs in a strongly connected directed graph $G$ on $n$ vertices that allows the nSSP equal to $2n-1$?  
\end{problem}

\begin{remark}
   \begin{enumerate}
   \item  In Theorem~\ref{thm:double_paths-first-m-loops} we will provide an example (for $m=1$) of a digraph whose underlying simple graph is a tree, with $2n-1$ arcs, which requires the nSSP. This guarantees that the boundary established in (\ref{numberarcs}) of Theorem~\ref{thm:does not allow} is sharp for digraphs whose underlying simple graph is a tree.
    \item   Let $C_{n,\loops}$ be a directed cycle with 
    $$E(C_{n,\loops})=\{(i,i+1) \!\! \mod n \colon i\in[n]\}\cup \loops.$$ 
    In \cite[Theorem 3.3]{2025-Saha-Tilis-VanderMeulen-VanTuyl-nSMP} authors proved that for all $n$ and all $\loops\subseteq [n]$ the directed graph $C_{n,\loops}$ requires the nSMP.  However, Theorem~\ref{thm:does not require}(\ref{2:not-require-no-double-arcs}) implies that $C_{n,\loops}$ does not require the nSSP. Moreover, if $|\loops|=1$, it does not allow the nSSP by Theorem~\ref{thm:does not allow}, but if $\loops=[n]$, it allows the nSSP.
    Hence $C_{n,\loops}$ is another family of directed graphs for which the nSSP and the nSMP behave differently.
    \end{enumerate}
\end{remark}

\section{Combinatorial property to ensure the nSSP}

In this section, we establish a combinatorial rule for directed graphs by extending the framework developed for undirected graphs in~\cite[Lemma~3.1]{20-Lin-SSPgraph}.

If $G$ is a directed graph and we want to show it requires the nSSP,  one has to start with an arbitrary matrix $A \in \calM(G)$ and $X\in \overline{\calM}(G^c)$ and show that the system $[A,X\trans]=O$ has only trivial solution. Equivalenty, we have to show that $X\in\overline{{\calM}}((K_n^{\circ})^c)$, where $K_n^{\circ}$ is the all looped complete digraph on $n=|V(G)|$ vertices.

Suppose that we know that $X\in \overline{\calM}(G_{\ell}^c)$ for some supergraph $G_{\ell}$ of $G$ on the same vertex set. Then the $i,j$-th equation of the $[A,X\trans]=O$ is equal to
\begin{align}0=[A,X\trans]_{i,j}&=\sum_{k\in[n]} A_{i,k}X_{k,j}\trans-\sum_{m\in[n]} X_{i,m}\trans A_{m,j}=\notag\\ 
&=\sum_{k\in[n]} A_{i,k}X_{j,k}-\sum_{m\in[n]} A_{m,j}X_{m,i}=\notag\\ 
&=\sum_{k \in N^+_G[i] \cap N^+_{G_{\ell}}[j]^c} A_{i,k}X_{j,k}-\sum_{m \in N^-_{G}[j]\cap N^-_{G_{\ell}}[i]^c} A_{m,j}X_{m,i}. \label{eq:ij-of-commutator}
\end{align}

In a special case, when expression~\eqref{eq:ij-of-commutator} is a monomial, i.e., consists only of one term, the assumption on $A_{x,y}\ne 0$ will imply that one entry of $X$ has to be $0$. This reduces the number of variables in the linear system $[A,X\trans]=O$ by one and sometimes a sequence of such observations will imply that $X=0$. This implies that the digraph $G$ requires the nSSP.

The next lemma provides conditions under which monomial equations show up, and so guarantees that some entries in matrix $X$ are equal to zero.

 \begin{lemma}\label{lem:rule1}
 Let $G$ be a digraph and let $G_{\ell}$ be a supergraph of $G$ on the same vertex set. Let $A \in \calM(G)$ and suppose that $X$ is such that $[A,X\trans]=O$, and $A_{\ell} \circ X=O$ for all matrices $A_{\ell} \in \calM(G_{\ell})$. 
 \begin{enumerate}[(A)]
     \item  \label{lem:rule1out}If for some  $i,j,k \in V(G)$ the conditions 
\begin{equation}\label{eq:monomial-k} N^+_G[i] \cap N^+_{G_{\ell}}[j]^c=\{k\} \text{ and } N^-_{G}[j]\cap N^-_{G_{\ell}}[i]^c=\emptyset
\end{equation}
hold, then $X\circ A_{\ell+1}=O$ for all $A_{\ell+1}\in \calM(G_{\ell+1})$, where $G_{\ell+1} = G_{\ell} + (j,k)$.  
\item \label{rule1in} If for some  $i,j,m \in V(G)$ the conditions 
\begin{equation}\label{eq:monomial-m}
 N^+_G[i] \cap N^+_{G_{\ell}}[j]^c=\emptyset  \text{ and } N^-_{G}[j]\cap N^-_{G_{\ell}}[i]^c=\{m\} 
\end{equation}
hold, then $X\circ A_{\ell+1}=O$ for all $A_{\ell+1}\in \calM(G_{\ell+1})$, where $G_{\ell+1} = G_{\ell} + (m,i)$. 
    \end{enumerate}
\end{lemma}
\begin{proof}
After putting~\eqref{eq:monomial-k} in~\eqref{eq:ij-of-commutator}, we observe that $[A,X\trans]_{i,j}=0$ simplifies to the monomial equation $$0=[A,X\trans]_{i,j}=A_{i,k}X_{j,k}$$
for $k\in N^+_G[i]$. Since $A_{i,k}\neq 0$ and so $X_{j,k}=0$. The statement follows.
Similarly~\eqref{eq:monomial-m} implies $$0=[A,X\trans]_{i,j}=-A_{m,j}X_{m,i}$$
for $m\in N^-_G[j]$. Since $A_{m,j}\neq 0$, we have $X_{m,i}=0$, and so the statement follows.
\end{proof}

 \begin{remark}\label{rem:symmetric-pattern} 
     Note that if $N_G^+[x]=N_G^-[x]$ for all $x\in V(G)$ (i.e, matrix $A\in \calM(G)$ has a symmetric pattern) and $N_{G_\ell}^+[x]=N_{G_\ell}^-[x]$ for all $x\in V(G)$ (i.e, matrix $X\in \overline{\calM}(G_{\ell}^c)$ has a symmetric pattern), then a triple $(i,j,k)$ satisfies condition~\eqref{eq:monomial-k} if and only if $(j,i,k)$ satisfies condition~\eqref{eq:monomial-m}. Therefore each condition~\eqref{eq:monomial-k} or~\eqref{eq:monomial-m} implies $X\circ A_{\ell+1}=O$ for all $A_{\ell+1}\in \calM(G_{\ell+1})$, where $G_{\ell+1} = G_{\ell} + (j,k)+ (k,j)= G_{\ell} + \{\{j,k\}\}$. 
 \end{remark}

Assume now that Lemma~\ref{lem:rule1} can be used iteratively in some directed graph $G$. Namely, assume that for $A \in \mathcal{M}(G)$ and $X$ such that $[A,X^T]=O$ and  $A\circ X=O$  next implications hold:
$$X \in \overline {\mathcal{M}}(G_1^c) \implies X \in \overline {\mathcal{M}}(G_2^c) \implies \cdots \implies X \in \overline {\mathcal{M}}(G_{k-1}^c)\implies X \in\overline {\mathcal{M}}((K_n^\circ)^c)$$ for some directed graphs $G_1\ldots, G_{k-1}$ where each $G_{i+1}$ is constructed from $G_i$ by adding some arcs using Lemma~\ref{lem:rule1}, and $K_n^{\circ}$ is the complete digraph with all loops. From the fact $X \in\overline {\mathcal{M}}((K_n^\circ)^c)$ it follows $X=O$, so $A$ has the nSSP. Thus we have proved the following Theorem which shows a combinatorial way of showing that a directed graph requires the nSSP.

\begin{theorem}\label{thm:rule1-nSSP}
  If there is a sequence of digraphs $G \subset G_1 \subset \ldots \subset G_k=K_n^{\circ}$ where $K_n^{\circ}$ is the complete digraph with all loops and each $G_{i+1}$ is constructed from $G_i$ by adding some arcs by Lemma~\ref{lem:rule1}, then $G$ requires the nSSP.
\end{theorem}

\begin{example}\label{ex:DH7}
  Let us show the strength of Lemma~\ref{lem:rule1} and Theorem~\ref{thm:rule1-nSSP} on the small digraph $G$ from Figure~\ref{fig:rule-1}. Let $A\in {\mathcal M}(G)$ and  $X=(x_{ij})\in \overline{\mathcal M}(G^c)$, such that $[A,X\trans]=O$.
  
\begin{figure}[h]
\begin{tikzpicture}[scale=0.7]
\foreach \i in {1,...,3} {
    \node[fill=white] (\i) at (90+120*\i:1) {};
   
};
\node[rectangle, draw=none] at (0.5,1) {$3$};
\node[rectangle, draw=none] at (1.5,-0.5) {$2$};
\node[rectangle, draw=none] at (-1.5,-0.5) {$1$};

\tikzset{every loop/.style={min distance=4mm,in=120,out=60,looseness=40}}
\foreach \i in {3} {
     \draw[thick,->] (\i) to[loop above] (\i);
};
 
\draw[->, thick] (1) to[bend left] (3);
\draw[->, thick] (3) to[bend left] (1);
\draw[->, thick] (2)--(1);
\draw[->, thick] (3)--(2);

\node[rectangle, draw=none] at (0,-1.5) {$G$};
\end{tikzpicture}\qquad
\begin{tikzpicture}[scale=0.7]
\foreach \i in {1,...,3} {
    \node[fill=white] (\i) at (90+120*\i:1) {};
 
};
\node[rectangle, draw=none] at (0.5,1) {$3$};
\node[rectangle, draw=none] at (1.5,-0.5) {$2$};
\node[rectangle, draw=none] at (-1.5,-0.5) {$1$};
\tikzset{every loop/.style={min distance=4mm,in=120,out=60,looseness=40}}
     \draw[thick,->] (3) to[loop above] (3);
\tikzset{every loop/.style={min distance=4mm,in=240,out=300,looseness=40}}
     \draw[thick,green,->] (1) to[loop above] (1);
 
\draw[->, thick] (1) to[bend left] (3);
\draw[->, thick] (3) to[bend left] (1);
\draw[->, thick] (2)--(1);
\draw[->, thick,green] (1) to[bend right] (2);
\draw[->, thick] (3)--(2);

\node[rectangle, draw=none] at (0,-1.5) {$G_1$};
\end{tikzpicture}
\qquad\begin{tikzpicture}[scale=0.7]
\foreach \i in {1,...,3} {
    \node[fill=white] (\i) at (90+120*\i:1) {};
};

\tikzset{every loop/.style={min distance=4mm,in=120,out=60,looseness=40}}
     \draw[thick,->] (3) to[loop above] (3);
\tikzset{every loop/.style={min distance=4mm,in=240,out=300,looseness=40}}
     \draw[thick,->] (1) to[loop above] (1);
 
\draw[->, thick] (1) to[bend left] (3);
\draw[->, thick] (3) to[bend left] (1);
\draw[->, thick] (2)--(1);
\draw[->, thick] (1) to[bend right] (2);
\draw[->, thick] (3)--(2);
\draw[->, thick,green] (2) to[bend right] (3);
\node[rectangle, draw=none] at (0.5,1) {$3$};
\node[rectangle, draw=none] at (1.5,-0.5) {$2$};
\node[rectangle, draw=none] at (-1.5,-0.5) {$1$};
\node[rectangle, draw=none] at (0,-1.5) {$G_2$};
\end{tikzpicture}
\qquad\begin{tikzpicture}[scale=0.7]
\foreach \i in {1,...,3} {
    \node[fill=white] (\i) at (90+120*\i:1) {};
  
};

\tikzset{every loop/.style={min distance=4mm,in=120,out=60,looseness=40}}
     \draw[thick,->] (3) to[loop above] (3);
\tikzset{every loop/.style={min distance=4mm,in=240,out=300,looseness=40}}
     \draw[thick,->] (1) to[loop above] (1);
      \draw[thick,->,green] (2) to[loop above] (2);
 \node[rectangle, draw=none] at (0.5,1) {$3$};
\node[rectangle, draw=none] at (1.5,-0.5) {$2$};
\node[rectangle, draw=none] at (-1.5,-0.5) {$1$};
\draw[->, thick] (1) to[bend left] (3);
\draw[->, thick] (3) to[bend left] (1);
\draw[->, thick] (2)--(1);
\draw[->, thick] (1) to[bend right] (2);
\draw[->, thick] (3)--(2);
\draw[->, thick] (2) to[bend right] (3);

\node[rectangle, draw=none] at (0,-1.5) {$G_3$};
\end{tikzpicture}
\caption{Sequence of $G\subset G_1 \subset G_2 \subset G_3=K_3^{\circ}$ which shows that $G$ requires the nSSP.}\label{fig:rule-1}
\end{figure}
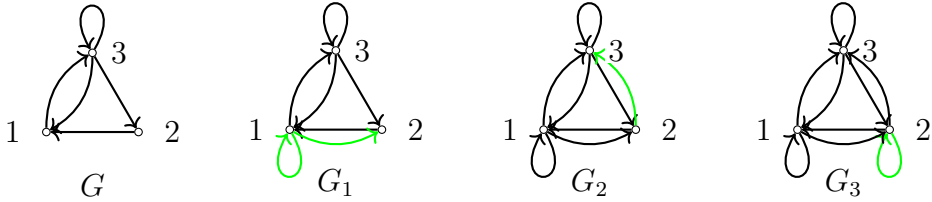

     Note that $N_G^+[i]\cap N_G^+[3]^c=\emptyset$ and $N_G^-[3]\cap N_G^-[i]^c=\{1\}$ for $i\in[2]$, and so by Lemma~\ref{lem:rule1}(\ref{rule1in}), it follows that $X\in \overline{{\mathcal M}}(G_1^c)$, where $G_1=G+\{(1,1),(1,2)\}$. Equivalently, $x_{11}=x_{12}=0$. Similarly
   $N_G^+[1]\cap N_{G_1}^+[2]^c=\{3\}$ and $N_G^-[2]\cap N_{G_1}^-[1]^c=\emptyset$, therefore $X\in \overline{{\mathcal M}}(G_2^c)$, where $G_2=G_1+\{(2,3)\}$ by the same Lemma. This implies  $x_{23}=0$. Finally, $N_G^+[3]\cap N_{G_2}^+[2]^c=\{2\}$ and $N_G^-[2]\cap N_{G_2}^-[3]^c=\emptyset$, therefore $X\in \overline{{\mathcal M}}(G_3^c)$, where $G_3=G_2+\{(2,2)\}$, by Lemma~\ref{lem:rule1}(\ref{lem:rule1out}). This implies $X\in \overline{\mathcal M}(G_3)$, and so by Theorem~\ref{thm:rule1-nSSP} $G$ requires the nSSP.
     The sequence of graphs $G\subseteq G_1 \subseteq G_2\subseteq G_3=K_3^{\circ}$ is shown on Figure~\ref{fig:rule-1}.
   \end{example}

\section{Graphs with pendent paths}

In this section, we apply Lemma~\ref{lem:rule1} to investigate directed graphs $G$ that contain pendent double paths. We demonstrate that any matrix $X \in \overline{\mathcal M}(G^c)$ commuting with an arbitrary matrix $A \in \mathcal M(G)$ must exhibit additional zero entries beyond those prescribed by its pattern. We present in Corollary~\ref{ex:lollipop} and Section~\ref{sec:double-paths} several examples of digraphs that require the nSSP and follow from the results of this section.

\subsection{Graphs with pendent looped double paths}
We  first show the following technical lemma, in which we considers graphs with pendent paths, in which all vertices on the path are looped.

\begin{lemma}\label{lem:induced-path-all-looped}
    Let $m,n\in \bN$, $m\leq n$, and let $G$ be a digraph with $V(G)=[n]$ such that its induced subgraph on vertices $[m]$ is $\dvec{P}_{m,[m]}$ such that $N_{G}[i]\subseteq [m]$ for all $i\in[m-1]$.

If $A\in \mathcal{M}(G)$ and $X=(x_{ij})\in \overline{\mathcal{M}}(G^c)$ such that $[A,X\trans]=O$, then $x_{ij}=0$ for all $1\leq i, j \leq m$.
    
    Moreover, if $N^+_G[m]=N^-_G[m]=\{m-1,m,m+1\}$ and $m+1\notin {\mathcal L}(G)$, then $x_{ij}=0$ for all $i\ne j$, $1\leq i, j \leq m+1$.
\end{lemma}

\begin{proof}
  Let $A=\begin{pmatrix}
       a_{i,j}
   \end{pmatrix}\in \mathcal{M}(G)$ and  $X=(x_{i,j})\in \overline{\mathcal{M}}(G^c)$, and let us define $a_{k,\ell}=x_{k,\ell}=0$ if $k\leq 0$ or $\ell \leq 0$. Observe first that 
   $N_G^+[i]=N_G^-[i]$ and $N_G^+[i]=\{i-1, i, i+1\}$ for all $i\in [m-1]$.
   Therefore 
   \begin{align}\label{eq:ij-loops}
       0=[A,X\trans]_{i,j}=a_{i,i-1}x_{j,i-1}+a_{i,i}x_{j,i}+a_{i,i+1}x_{j,i+1} -a_{j+1,j}x_{j+1,i}-a_{j,j}x_{j,i}-a_{j-1,j}x_{j-1,i}.
   \end{align}
  
   In particular, when $j=i+1$, it follows that
   $0=[A,X\trans]_{1,2}=-a_{3,2}x_{3,1}$ and 
   \begin{equation}\label{eq:dist2}
       0=[A,X\trans]_{i,i+1}=a_{i,i-1}x_{i+1,i-1}-a_{i+2,i+1}x_{i+2,i}
   \end{equation}
   for all $2 \leq i \leq m-2$. Since $a_{j+1,j}\ne 0$ for $j\in [m-2]$, it follows inductively that for all $i\in [m-2]$ we have $x_{i+2,i}=0$. By Remark~\ref{rem:symmetric-pattern}, $x_{i,i+2}=0$  as well. 
   
   Therefore $X\in \overline{\mathcal{M}}(G_2^c)$ for $G_2=G+\{\{i,i+2\}\colon i\in [m-2]\}$. For all $k$, $2\leq k\leq m-1$ let us define $G_k=G+\{\{i,i+j\}\colon j\in [k], i\in [m-j]\}$ and $G_1=G$. We will show that if $X\in \overline{\mathcal{M}}(G_{k}^c)$, then $X\in \overline{\mathcal{M}}(G_{k+1}^c)$ for all $k\in [m-2]$. 

   We proved the statement is true for $k=1$. Let $X\in \overline{\mathcal{M}}(G_{k}^c)$, i.e.~$x_{i,i+j}=0$ for all $i\in [m-k]$ and $j\in [k]$. Therefore $$0=[A,X\trans]_{1,k+1}=-a_{k+2,k+1}x_{k+2,1}$$ and so $x_{k+2,1}=0$
   and for all $2\leq i \leq m-k-1$, equation~\eqref{eq:ij-loops} simplifies to
   \begin{align}\label{eq:dist-k}
       0&=[A,X\trans]_{i,i+k}=\notag\\
       &=a_{i,i-1}x_{i+k,i-1}+a_{i,i}x_{i+k,i}+a_{i,i+1}x_{i+k,i+1} \notag\\
       &\qquad \qquad-a_{i+k+1,i+k}x_{i+k+1,i}-a_{i+k,i+k}x_{i+k,i}-a_{i+k-1,i+k}x_{i+k-1,i}=\notag\\
       &=a_{i,i-1}x_{i+k,i-1}-a_{i+k+1,i+k}x_{i+k+1,i}
   \end{align}
 Since $x_{k+2,1}=0$, equations~\eqref{eq:dist-k} imply that for all $i\in [m-k-1]$ we have $x_{i+k+1,i}=0$. By Remark~\ref{rem:symmetric-pattern} it follows that $x_{i,i+k+1}=0$, and so $X\in \overline{\mathcal{M}}(G_{k+1}^c)$.
This inductively implies that $X\in \overline{\mathcal{M}}(G_{m-1}^c)$  and so $x_{i,j}=0$ for all $i,j \in [m]$.

  For the moreover case observe that~\eqref{eq:dist2} and~\eqref{eq:dist-k} both extend to $i=m-1$ and $i=m-k$, respectively, which completes the proof.
\end{proof}

The next corollary shows an example of the use of Lemma~\ref{lem:induced-path-all-looped} to show that the family of all double lollipop digraphs with a specific assignment of loops requires the nSSP.

\begin{corollary}\label{ex:lollipop}
Let $G=\dvec{L}_{p,k,\loops}$ be a double lollipop digraph on $n=p+k$ vertices, where $V(G)=[n]$ and  
$$E(G)=\{\{i,i+1\}\colon i\in [p]\}\cup \{\{i,j\}\colon  p+1\leq i< j \leq n\}\cup \{(\ell,\ell)\colon \ell\in\loops\}.$$

If $[p+2]\subseteq \loops$, then $\dvec{L}_{p,k,\loops}$ requires the nSSP.
\end{corollary}

\begin{proof}
 Let $A \in \mathcal{M}(G)$ and $X \in \overline{\mathcal{M}}(G^c)$ such that $[A,X\trans]=O$. Observe that the induced subgraph on vertices $[p+1]$ is $\dvec{P}_{p+1,[p+1]}$ and $N_G[i]\subseteq [p+1]$ for all $i\in [p]$. By Lemma~\ref{lem:induced-path-all-looped} it follows that $x_{ij}=0$ for all $1\leq i,j\leq p+1$.
 Therefore $X\in \overline{\mathcal{M}}(G_1^c)$ for $G_1=G+\{\{i,j\}: 1\leq i, j \leq p+1, \ i \neq j\}$.

 For every $j\notin \loops$ we have $N^+_G[j]\cap N^+_{G_1}[p+2]^c=\emptyset$ and $N^-_G[p+2]\cap N^-_{G_1}[j]^c=\{j\}$ and by Lemma~\ref{lem:rule1} we also have  $x_{j,j}=0$ for all $j\notin \mathcal{L}$. Thus $X\in \overline{\mathcal{M}}(G_2^c)$ for $G_2=G_1+\{(\ell,\ell)\colon  \ell \notin \loops\}$. For $G=\dvec{L}_{5,6,[8]}$ see the digraphs $G_1$ and $G_2$ on Figure~\ref{fig:ex-lollipop-5-6-8-G1-G2}.

 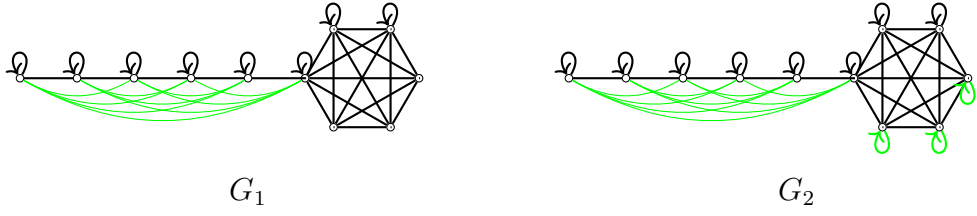
\begin{figure}[htb]
\begin{tikzpicture}[scale=0.75]
\foreach \i in {1,...,6} {
    \pgfmathsetmacro{\angle}{60 * (\i - 1)}
    \node[fill=white] (\i) at (\angle:1) {};
};
\foreach \i in {1,...,6} {
\foreach \j in {\i,...,6} {
 \draw[thick] (\i) -- (\j);
}};
\foreach \i in {7,...,11} {
     \draw[thick] (\i-13,0) -- (\i-12,0);
    \node[fill=white] (\i) at (\i-13,0) {};
};
 \tikzset{every loop/.style={min distance=4mm,in=120,out=60,looseness=20}}
\foreach \i in {2,3,4,7,8,9,10,11} {
\draw[thick,->] (\i) to[loop above] (\i);};
\tikzset{every loop/.style={min distance=4mm,in=240,out=300,looseness=20}}
\draw[green] (4) to[bend left] (10);
\draw[green] (4) to[bend left] (9);
\draw[green] (4) to[bend left] (8);
\draw[green] (4) to[bend left] (7);
\draw[green] (11) to[bend left] (9);
\draw[green] (11) to[bend left] (8);
\draw[green] (11) to[bend left] (7);
\draw[green] (10) to[bend left] (8);
\draw[green] (10) to[bend left] (7);
\draw[green] (9) to[bend left] (7);

\node[rectangle, draw=none] at (-2,-2) {$G_1$};
\end{tikzpicture}\qquad\qquad
\begin{tikzpicture}[scale=0.75]
\foreach \i in {1,...,6} {
    \pgfmathsetmacro{\angle}{60 * (\i - 1)}
    \node[fill=white] (\i) at (\angle:1) {};
};
\foreach \i in {1,...,6} {
\foreach \j in {\i,...,6} {
 \draw[thick] (\i) -- (\j);
}};
\foreach \i in {7,...,11} {
     \draw[thick] (\i-13,0) -- (\i-12,0);
    \node[fill=white] (\i) at (\i-13,0) {};
};
 \tikzset{every loop/.style={min distance=4mm,in=120,out=60,looseness=20}}
\foreach \i in {2,3,4,7,8,9,10,11} {\draw[->,thick] (\i) to[loop above] (\i);};
\tikzset{every loop/.style={min distance=4mm,in=240,out=300,looseness=20}}
\foreach \i in {1,5,6} {
\draw[thick, green,->] (\i) to[loop below] (\i);};
\draw[green] (4) to[bend left] (10);
\draw[green] (4) to[bend left] (9);
\draw[green] (4) to[bend left] (8);
\draw[green] (4) to[bend left] (7);
\draw[green] (11) to[bend left] (9);
\draw[green] (11) to[bend left] (8);
\draw[green] (11) to[bend left] (7);
\draw[green] (10) to[bend left] (8);
\draw[green] (10) to[bend left] (7);
\draw[green] (9) to[bend left] (7);

\node[rectangle, draw=none] at (-2,-2) {$G_2$};
\end{tikzpicture}
\caption{Supergraphs $G_1$ and $G_2$ of the lollipop digraph $G=\dvec{L}_{5,6,[8]}$ from the proof of Corollary~\ref{ex:lollipop}. In black, the double lollipop digraph $\dvec{L}_{5,6,[8]}$ shown, and all undirected edges represent double arcs.}\label{fig:ex-lollipop-5-6-8-G1-G2}
\end{figure}

For $i=2,3,\dots,p+1$ let us define $G_{i+1}:=G_i+\{\{p+2-i,j\} \colon j> p+1\}$. We will prove the claim that if $X\in \overline{\mathcal{M}}(G_i^c)$ for some $i\leq p$, then $X\in \overline{\mathcal{M}}(G_{i+1}^c)$.

For $i=2$  observe that $N^+_{G_2}[p+1]=N^-_{G_2}[p+1]=V(G)$ and so $N_G^+[j]\cap N^+_{G_2}[p+1]^c=\emptyset$ for all $j>p+1$. Also, $N_G^-[p+1]\cap N^-_{G_2}[j]^c=\{p\}$, so it follows that $x_{p,j}=x_{j,p}=0$, and so $X\in \overline{\mathcal{M}}(G_3^c)$. Hence the claim is true for $i=2$.
Assume now that $X\in \overline{\mathcal{M}}(G_i^c)$ for some $i\geq 2$. Then $N^+_{G_i}[p+3-i]=N^-_{G_i}[p+3-i]=V(G)$  and so $N_G^+[j]\cap N^+_{G_i}[p+3-i]^c=\emptyset$ and $N_G^-[p+3-i]\cap N^-_{G_i}[j]^c=\{p+2-i\}$ for all $j>p+1$. By Lemma~\ref{lem:rule1} and Remark~\ref{rem:symmetric-pattern} it follows that $x_{p+2-i,j}=x_{j,p+2-i}=0$, and so $X\in \overline{\mathcal{M}}(G_{i+1}^c)$.

Now, inductively we have $X\in \overline{\mathcal{M}}(G_{p+2}^c)$. Since $G_{p+2}=K_{n}^{\circ}$, it follows that $X=O$ and so $G$ requires the nSSP.
 \end{proof}
 
 \begin{example}
   Observe that if $\loops=[p+1]$, then $\dvec{L}_{p,k,\loops}$ does not necessarily require the nSSP. For example, for matrices $$A=\left(
\begin{array}{ccccc}
 -1 & 1 & 0 & 0 & 0 \\
 1 & 1 & -2 & 1 & 1 \\
 0 & 1 & 0 & -2 & 1 \\
 0 & 1 & -2 & 0 & 1 \\
 0 & 1 & 4 & -5 & 0 \\
\end{array}
\right)\in {\mathcal{M}}(\dvec{L}_{1,4,[2]})
 \text{ and }
 X=\left(
\begin{array}{ccccc}
 0 & 0 & 1 & 1 & 1 \\
 0 & 0 & 0 & 0 & 0 \\
 -2 & 0 & -1 & 0 & 0 \\
 1 & 0 & 0 & -1 & 0 \\
 1 & 0 & 0 & 0 & -1 \\
\end{array}
\right)$$ 
we have $A\circ X=[A,X\trans]=O$, and so $A$ does not have the nSSP.
 \end{example}

 \subsection{Graphs with pendent loopless double paths}
 
 Next, we present another lemma concerning pendent paths in which all vertices on the path are non-looped.
 
\begin{lemma}\label{lem:induced-path-loopless}
  Let $m,n\in \bN$, $m< n$, and let $G$ be a digraph with $V(G)=[n]$ such that its induced subgraph on vertices $[m]$ is $\dvec{P}_{m,\emptyset}$ such that $N_{G}[i]\subseteq [m]$ for all $i\in[m-1]$. 
 If $A\in \mathcal{M}(G)$ and $X=(x_{ij})\in \overline{\mathcal{M}}(G^c)$ such that $[A,X\trans]=O$, then $x_{ij}=0$ for all $1\leq i, j \leq m$, $|i-j|$ is odd.

    Moreover, if $N^+_G[m]=N^-_G[m]=\{m-1,m+1\}$ and $m+1\in {\mathcal L}(G)$, then $x_{ij}=0$ for all $1\leq i, j \leq m+1$, $|i-j|$ is odd.
    
\end{lemma}

\begin{proof}
Let $A=\begin{pmatrix}
       a_{i,j}
   \end{pmatrix}\in \mathcal{M}(G)$,  $X=(x_{i,j})\in \overline{\mathcal{M}}(G^c)$, and let us define $a_{k,\ell}=x_{k,\ell}=0$ if $k\leq 0$ or $\ell \leq 0.$  Observe first that 
   $N_G^+[i]=N_G^-[i]$ and $N_G^+[i]=\{i-1, i+1\}$ for all $i\in [m-1]$. 
   Therefore, for all $i,j \in [m-1]$ we have
   \begin{align}\label{eq:i-j-loops-case}
       0=[A,X\trans]_{i,j}=&a_{i,i-1}x_{j,i-1}+a_{i,i+1}x_{j,i+1}-a_{j+1,j}x_{j+1,i}-a_{j-1,j}x_{j-1,i}.
   \end{align}
   In particular, when $j=i+2$, it follows that $0=[A,X\trans]_{1,3}=-a_{4,3}x_{4,1}=0$ and $$0=[A,X\trans]_{i,i+2}=a_{i,i-1}x_{i+2,i-1}-a_{i+3,i+2}x_{i+3,i}$$
   for all $i$, $2 \leq i \leq m-3$. Since $a_{j+1,j}\ne 0$ for $j\in [m-2]$, it follows inductively that for all $i\in [m-3]$ we have $x_{i+3,i}=0$. By Remark~\ref{rem:symmetric-pattern}, $x_{i,i+3}=0$  as well.

 Therefore $X\in \overline{\mathcal{M}}(G_3^c)$ for $G_3=G+\{\{i,i+3\}\colon i\in [m-3]\}$. For all odd $k$, $3\leq k\leq m-1$ let us define $G_k=G+\{\{i,i+j\}\colon j\in \{1,3,5,\ldots,k\}, i\in [m-j]\}$ and $G_1=G$. We will show that if $X\in \overline{\mathcal{M}}(G_{k}^c)$, then $X\in \overline{\mathcal{M}}(G_{k+2}^c)$ for all odd $k\in [m-3]$.

   We proved the statement is true for $k=1$. Let $X\in \overline{\mathcal{M}}(G_{k}^c)$ for some odd $k$, i.e.~$x_{i,i+j}=0$ for all $i\in [m-k]$ and all odd $j\in [k]$. Equation~\eqref{eq:i-j-loops-case} simplifies to 
   \begin{align}
       \label{eq:1-k+1}
   0=[A,X\trans]_{1,k+2}&=a_{1,2}x_{k+2,2}-a_{k+3,k+2}x_{k+3,1}-a_{k+1,k+2}x_{k+1,1}=-a_{k+3,k+2}x_{k+3,1}
   \end{align}
   and 
   \begin{align}\label{eq:dist-k-odd}
       0&=[A,X\trans]_{i,i+k+1}=\notag\\
       &=a_{i,i-1}x_{i+k+1,i-1}+a_{i,i+1}x_{i+k+1,i+1}-a_{i+k+2,i+k+1}x_{i+k+2,i}-a_{i+k,i+k+1}x_{i+k,i}=\notag \\
       &=a_{i,i-1}x_{i+k+1,i-1}-a_{i+k+1+1,i+k+1}x_{i+k+2,i}       
   \end{align}
   and for all $2\leq i \leq m-k-1$. By~\eqref{eq:1-k+1} and Remark~\ref{rem:symmetric-pattern} it follows that $x_{k+3,1}=x_{1,k+3}=0$. Therefore, equations~\eqref{eq:dist-k-odd} imply that for all $i\in [m-k-1]$ we have $x_{i+k+2,i}=0$. By Remark~\ref{rem:symmetric-pattern} it follows that $x_{i,i+k+2}=0$, and so $X\in \overline{\mathcal{M}}(G_{k+2}^c)$.
This inductively implies that $X\in \overline{\mathcal{M}}(G_{K}^c)$ for the largest odd number $K\in [m]$ and so $x_{i,j}=0$ for all $i,j \in [m]$ if $|i-j|$ is odd.

Again, for the moreover case observe that~\eqref{eq:i-j-loops-case} extends to $i,j \in [m]$, which completes the proof.
\end{proof}

\begin{example}\label{example:doublepath}
Consider a digraph $\dvec{P}_{10,[3]}$. Let $A\in \mathcal{M}(G)$ and let $X=(x_{i,j})$ be such that $A\circ X=[A,X\trans]=O$. Then 
$$X=\begin{pmatrix}
 0 & 0 & {\color{blue} 0} & {\color{blue} 0} & x_{1,5} & x_{1,6} & x_{1,7} & x_{1,8} & x_{1,9} & x_{1,10} \\
 0 & 0 & 0 & {\color{blue} 0} & x_{2,5} & x_{2,6} & x_{2,7} & x_{2,8} & x_{2,9} & x_{2,10} \\
 {\color{blue} 0} & 0 & 0 & 0 & x_{3,5} & {\color{magenta} 0} & x_{3,7} & {\color{magenta} 0} & x_{3,9} & {\color{magenta} 0} \\
 {\color{blue} 0} & {\color{blue} 0} & 0 & x_{4,4} & 0 & x_{4,6} & {\color{magenta} 0} & x_{4,8} & {\color{magenta} 0} & x_{4,10} \\
 x_{5,1} & x_{5,2} & x_{5,3} & 0 & x_{5,5} & 0 & x_{5,7} & {\color{magenta} 0} &  x_{5,9} & {\color{magenta} 0} \\
x_{6,1} & x_{6,2} & {\color{magenta} 0} & x_{6,4} & 0 & x_{6,6} & 0 &  x_{6,8} & {\color{magenta} 0} &  x_{6,10} \\
 x_{7,1} & x_{7,2} & x_{7,3} &{\color{magenta} 0} & x_{7,5} & 0 &  x_{7,7} & 0 &  x_{7,9} & {\color{magenta} 0} \\
 x_{8,1} & x_{8,2} & {\color{magenta} 0} & x_{8,4} & {\color{magenta} 0} &  x_{8,6} & 0 &  x_{8,8} & 0 &  x_{8,10} \\
 x_{9,1} & x_{9,2} & x_{9,3} & {\color{magenta} 0} & x_{9,5} & {\color{magenta} 0} &  x_{9,7} & 0 &  x_{9,9} & 0 \\
 x_{10,1} & x_{10,2} & {\color{magenta} 0} & x_{10,4} & {\color{magenta} 0} &  x_{10,6} & {\color{magenta} 0} &  x_{10,8} & 0 &  x_{10,10} \\
\end{pmatrix},$$
where by Lemma~\ref{lem:induced-path-all-looped} it follows that $x_{i,j}=0$ for all $\{i,j\}\in [4]\times [4]\setminus \{4,4\}$ (denoted by blue in matrix $X$) and by Lemma~\ref{lem:induced-path-loopless} it follows that $x_{i,i+3}=x_{i+3,i}=x_{j,j+5}=x_{j+5,j}=x_{3,10}=x_{10,3}=0$ for all $3\leq i\leq 7$ and $3\leq j\leq 5$ (denoted by pink in $X$). 

In Theorem~\ref{thm:double_paths-first-m-loops} we will show that the equation $[A,X\trans]=O$ forces also all other entries of $X$ to be $0$.
\end{example}

\section{Double paths}\label{sec:double-paths}

In this section, we investigate Problem \ref{problem-6.5}, and 
identify several classes of loop assignments $\loops \subseteq [n]$ on double paths, for which the corresponding graph $P_{n,\loops}$ either requires the nSSP, allows but does not require the nSSP, or does not allow the nSSP. The following motivating example highlights the importance of the digraphs that require the nSSP.

\begin{example}
 In~\cite{2024-Breen-Allow-sequence} the authors prove that matrices $H_1$, $H_2$, $H_3$, $H_4$, $H_7$ and $H_{13}$ have the nSSP. In Example~\ref{ex:DH7} we showed that the corresponding ${\mathcal D}(H_7)$, shown on Figure~\ref{fig:D1-13-from-Breen}, requires the nSSP.  
 All other digraphs ${\mathcal D}(H_1)$, ${\mathcal D}(H_2)$, ${\mathcal D}(H_3)$, ${\mathcal D}(H_4)$ and ${\mathcal D}(H_{13})$, shown on Figure~\ref{fig:D1-13-from-Breen}, are double paths and we will show  in Theorems~\ref{thm:double_paths-first-m-loops} and~\ref{thm:double-path-vtx-2} that they all require the nSSP. In particular, all their sign patterns require the nSSP. Therefore, when selecting matrices $H_1$, $H_2$, $H_3$, $H_4$, $H_7$ and $H_{13}$ with the desired spectrum, the nSSP is automatically guaranteed. 

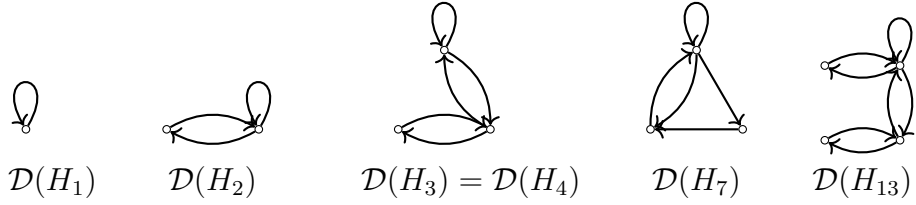
\begin{figure}[h]
\begin{tikzpicture}[scale=0.7]
 \node[fill=white] (1) at (0,-0.5) {};
 
\tikzset{every loop/.style={min distance=4mm,in=120,out=60,looseness=40}}
\foreach \i in {1} {
     \draw[thick,->] (\i) to[loop above] (\i);
};
\node[rectangle, draw=none] at (0.5,-1.5) {${\mathcal D}(H_1)$};
\end{tikzpicture}\qquad
\begin{tikzpicture}[scale=0.7]
\foreach \i in {1,2} {
    \node[fill=white] (\i) at (90+120*\i:1) {};
   
};

\tikzset{every loop/.style={min distance=4mm,in=120,out=60,looseness=40}}
\foreach \i in {2} {
     \draw[thick,->] (\i) to[loop above] (\i);
};

\draw[->, thick] (1) to[bend left] (2);
\draw[->, thick] (2) to[bend left] (1);

\node[rectangle, draw=none] at (0,-1.5) {${\mathcal D}(H_2)$};
\end{tikzpicture}\qquad
\begin{tikzpicture}[scale=0.7]
\foreach \i in {1,...,3} {
    \node[fill=white] (\i) at (90+120*\i:1) {};
  
};

\tikzset{every loop/.style={min distance=4mm,in=120,out=60,looseness=40}}
\foreach \i in {3} {
     \draw[thick,->] (\i) to[loop above] (\i);
};

\draw[->, thick] (1) to[bend left] (2);
\draw[->, thick] (2) to[bend left] (1);
\draw[->, thick] (2) to[bend left] (3);
\draw[->, thick] (3) to[bend left] (2);

\node[rectangle, draw=none] at (0.5,-1.5) {${\mathcal D}(H_3)={\mathcal D}(H_4)$};
\end{tikzpicture}\qquad
\begin{tikzpicture}[scale=0.7]
\foreach \i in {1,...,3} {
    \node[fill=white] (\i) at (90+120*\i:1) {};
 
};

\tikzset{every loop/.style={min distance=4mm,in=120,out=60,looseness=40}}
\foreach \i in {3} {
     \draw[thick,->] (\i) to[loop above] (\i);
};
 
\draw[->, thick] (1) to[bend left] (3);
\draw[->, thick] (3) to[bend left] (1);
\draw[->, thick] (2)--(1);
\draw[->, thick] (3)--(2);

\node[rectangle, draw=none] at (0,-1.5) {${\mathcal D}(H_7)$};
\end{tikzpicture}\qquad
\begin{tikzpicture}[scale=0.7]
\foreach \i in {1,...,4} {
    \node[fill=white] (\i) at (135+90*\i:1) {};
  
};

\tikzset{every loop/.style={min distance=4mm,in=120,out=60,looseness=40}}
\foreach \i in {3} {
     \draw[thick,->] (\i) to[loop above] (\i);
};

\draw[->, thick] (1) to[bend left] (2);
\draw[->, thick] (2) to[bend left] (1);
\draw[->, thick] (2) to[bend left] (3);
\draw[->, thick] (3) to[bend left] (2);
\draw[->, thick] (4) to[bend left] (3);
\draw[->, thick] (3) to[bend left] (4);

\node[rectangle, draw=none] at (0,-1.5) {${\mathcal D}(H_{13})$};
\end{tikzpicture}
\caption{Five examples of digraphs from~\cite{2024-Breen-Allow-sequence} on at most four vertices, all requiring the nSSP.}\label{fig:D1-13-from-Breen}
\end{figure}
\end{example}

We start with a technical lemma that shows that the commuting condition $[A,X\trans]=O$ on some tridiagonal matrices $A$ force many entries in matrix $X$ of the complementary pattern to be zero.

\begin{lemma}\label{lem:pendent-looped-path-on-a-path}
 Let $m,n\in \bN$ with $m< n$, and let $G:=\dvec{P}_{n,\loops}$, $[m] \subseteq \loops$ and $m+1\notin \loops$. If $A\in{\mathcal M}(G)$, $X=(x_{ij}) \in \overline{\mathcal{M}}(G^c)$ and  $[A,X\trans]=0$, then $x_{ij}=0$ for all $i+j\leq 2m+1$.  
\end{lemma}

\begin{proof}
Since the vertex $m+1$ is the first non-looped vertex on the path,  using Lemma~\ref{lem:induced-path-all-looped}  we have $X \in \overline{\mathcal{M}}(G_1^c)$, where $G_{1}=G+\{\{i,j\}: 1\leq i, j \leq m+1, i \neq j\}$. 
For every  $i\in [m-1]$ we have 
$N^{+}_G[m+1]\cap N^{+}_{G_1}[i]^c=\{m+2\}$ and $N^{-}_G[i]\cap N^{-}_{G_1}[m+1]^c=\emptyset$,  
and so by Lemma~\ref{lem:rule1} and Remark~\ref{rem:symmetric-pattern} we have $x_{i,m+2}=x_{m+2,i}=0$ for all $i\in [m-1]$, or equivalently 
 $X\in \overline{\mathcal{M}}(G_{2}^c)$ for $G_{2}=G_1+\{\{i,m+2\}\colon  i\in [m-1]\}$.

Let $M:=\min\{m+1,n-m\}$ and for $k=2,\dots,M$ let us define $$G_{k}:=G_{k-1}+\{\{j,m+k\} \colon j\in [m-k+1]\}.$$ If $X\in \overline{\mathcal{M}}(G_{k}^c)$ for some $k\leq M-1$, then $x_{i,m+k}=0$ for all $i\leq m-k+1$. Hence
for all $j\in [m-k-1]$ we have 
$N^{+}_G[j]\cap N^{+}_{G_{k}}[m+k]^c=\emptyset$ and $N^{-}_G[m+k]\cap N^{-}_{G_{k}}[j]^c=\{m+k+1\}$.
Therefore by Lemma~\ref{lem:rule1} and Remark~\ref{rem:symmetric-pattern} it follows that $X\in \overline{\mathcal{M}}(G_{k+1}^c)$. Since we proved that $X\in \overline{\mathcal{M}}(G_{2}^c)$, it follows that $X\in \overline{\mathcal{M}}(G_M^c)$ and so $x_{ij}=0$ for all $i+j\leq 2m+1$.
\end{proof}

\begin{theorem}\label{thm:double_paths-lots-of-loops}  
  Let $n\in \bN$ and choose $\loops\subseteq [n]$. Let $m$ be maximal integer, such that $[m] \subseteq \loops$. If $m>\frac{1}{2}(n-1)$, then $\dvec{P}_{n,\loops}$ requires the nSSP.  
\end{theorem}
\begin{proof}
    Let $A \in \mathcal{M}(G)$ be arbitrary and $X=(x_{ij}) \in \overline{\mathcal{M}}(G^c)$ such that $[A,X\trans]=0$. By Lemma~\ref{lem:pendent-looped-path-on-a-path} it follows that $x_{ij}=0$ for all $i+j\leq 2m+1$. Let $k:=2m-n+1>0$ and observe that the first $k$ columns of $X$  and the first $k$ rows of $X$ are equal to $\bf 0$.

Note that for every $j\in\bN$, $0\leq j\leq 2n-2m-1$ and for each $i\in \{0,1,\ldots,j\}$ we have
\begin{align}\label{eq:large-m}
        [A,X\trans]_{n-i,k+j}&=a_{n-i,n-i-1}x_{k+j,n-i-1}+a_{n-i,n-i}x_{k+j,n-i}+a_{n-i,n-i+1}x_{k+j,n-i+1} \notag \\
        &\qquad -a_{k+j-1,k+j}x_{k+j-1,n-i}-a_{k+j,k+j}x_{k+j,n-i}-a_{k+j+1,k+j}x_{k+j+1,n-i}.
\end{align}

    In particular, since $x_{\ell,t}=0$ for all $\ell \leq k$, it follows using~\eqref{eq:large-m} for $i=j=0$, that
    $$[A,X\trans]_{n,k}=-a_{k+1,k}x_{k+1,n}.
    $$
    This implies $x_{k+1,n}=0$ and therefore the first $k+1$ rows of $X$ are equal to $\bf 0$. By symmetry also the first $k+1$ columns of $X$ are equal to $\bf 0$. 
    
    Now assume that the first $k+j$ columns and rows of $X$ are equal to $\bf 0$ for some $0\leq j\leq 2n-2m-1$. By~\eqref{eq:large-m} it follows that
    \begin{align*}
        [A,X\trans]_{n-i,k+j}&
        =a_{k+j+1,k+j}x_{k+j+1,n-i}
    \end{align*}
    for each $i\in \{0,1,\ldots,j\}$, which implies that $x_{k+j+1,n-i}=0$. By symmetry $x_{n-i,k+j+1}=0$ as well, and so the first $k+j+1$ columns and rows of $X$ are equal to $\bf 0$. Since we proved this statement for $j=0$, it follows that $X=O$. 
\end{proof}

In Theorem~\ref{thm:double-path-vtx-2}(\ref{item:without-middle-loop}) we show that Theorem~\ref{thm:double_paths-lots-of-loops} might not be true in the case $m=\frac{1}{2}(n-1)$. Also, in Example~\ref{ex:n<=5} we show that for $n=5$, $\dvec{P}_{5,\{1,2,4\}}$ requires the nSSP and $\dvec{P}_{5,\{1,2,5\}}$ allows but does not require the nSSP.
However, the next Theorem shows that if some initial segment of vertices has loops and the remaining vertices do not (i.e., $\dvec{P}_{5,[1]}$, $\dvec{P}_{5,[2]}$, $\dvec{P}_{5,[3]}$, $\dvec{P}_{5,[4]}$, or  $\dvec{P}_{5,[5]}$), then the double path still requires the nSSP. First, we illustrate the idea of its proof by an example.

\begin{example}
Let us again consider $G=\dvec{P}_{10,[3]}$,  \(A \in \mathcal{M}(G)\) and  \(X \in \overline{\mathcal{M}}(G^c)\) as in the Example~\ref{example:doublepath}. 
We already know from the example that several entries of $$X=
\begin{pmatrix}
 0 & 0 & \color{blue}x_{1,3} & \color{blue}x_{1,4} & \color{blue}x_{1,5} &\color{blue} x_{1,6} & \color{green}x_{1,7} & \color{orange}x_{1,8} &\color{red} x_{1,9} &\color{OliveGreen} x_{1,10} \\
 0 & 0 & 0 & \color{blue}x_{2,4} &\color{blue} x_{2,5} & \color{green}x_{2,6} &\color{orange} x_{2,7} &\color{red} x_{2,8} & \color{OliveGreen}x_{2,9} &\color{cyan} x_{2,10} \\
 \color{blue}x_{3,1} & 0 & 0 & 0 & \color{green}x_{3,5} &\color{magenta} x_{3,6} &\color{red} x_{3,7} &\color{magenta} x_{3,8} & \color{cyan}x_{3,9} & \color{magenta}x_{3,10} \\
 \color{blue}x_{4,1} & \color{blue}x_{4,2} & 0 & \color{green}x_{4,4} & 0 &\color{red} x_{4,6} &\color{magenta} x_{4,7} &\color{cyan} x_{4,8} & \color{magenta}x_{4,9} &\color{gray}  x_{4,10} \\
 \color{blue}x_{5,1} & \color{blue}x_{5,2} & \color{green}x_{5,3} & 0 & \color{red}x_{5,5} & 0 & \color{cyan}x_{5,7} &\color{magenta} x_{5,8} & \color{gray} x_{5,9} &\color{magenta} x_{5,10} \\
 \color{blue}x_{6,1} & \color{green}x_{6,2} & \color{magenta}x_{6,3} & \color{red}x_{6,4} & 0 &\color{cyan} x_{6,6} & 0 & \color{gray} x_{6,8} & \color{magenta}x_{6,9} & \color{gray} x_{6,10} \\
 \color{green}x_{7,1} & \color{orange}x_{7,2} &\color{red} x_{7,3} & \color{magenta}x_{7,4} &\color{cyan} x_{7,5} & 0 & \color{gray} x_{7,7} & 0 & \color{gray} x_{7,9} &\color{magenta} x_{7,10} \\
 \color{orange}x_{8,1} & \color{red}x_{8,2} &\color{magenta} x_{8,3} & \color{cyan}x_{8,4} &\color{magenta} x_{8,5} & \color{gray} x_{8,6} & 0 & \color{gray} x_{8,8} & 0 & \color{gray} x_{8,10} \\
 \color{red}x_{9,1} &\color{OliveGreen} x_{9,2} & \color{cyan}x_{9,3} & \color{magenta}x_{9,4} &\color{gray} x_{9,5} & \color{magenta}x_{9,6} & \color{gray} x_{9,7} & 0 & \color{gray} x_{9,9} & 0 \\
 \color{OliveGreen}x_{10,1} &\color{cyan} x_{10,2} & \color{magenta}x_{10,3} & \color{gray}x_{10,4} & \color{magenta}x_{10,5} & \color{gray} x_{10,6} & \color{magenta}x_{10,7} & \color{gray} x_{10,8} & 0 & \color{gray} x_{10,10}
\end{pmatrix}$$
must be equal to $0$, namely  all entries coloured in pink and some blue ones. 
However, we proved in Lemma~\ref{lem:pendent-looped-path-on-a-path}, that all blue entries are forced to be $0$.
In the proof of the Theorem~\ref{thm:double_paths-first-m-loops} these entries will correspond to the digraphs \(\color{magenta} G_1\) and \({\color{blue} G_2}\). 

All the remaining entries of \(X\) are here coloured according to their antidiagonals (i.e., by the value of \(i+j\)), assigning a unique colour to each corresponding supergraph of \(G\), 
 namely 
$\color{green} G_3$, $\color{orange} G_4$, $\color{red} G_5$, $\color{OliveGreen} G_6$, $\color{cyan} G_7$,
respectively. 
The last step of the proof shows that also all the remaining entries of $X$ (coloured gray), which correspond to the digraph
$\color{gray} G_9 = K_{10}^{\circ}$, are equal to zero. 
\end{example}

\begin{theorem}\label{thm:double_paths-first-m-loops}
 For every $m,n\in \bN$, $m\leq n$, the double path $\dvec{P}_{n,\loops}$ with $V(\dvec{P}_{n,\loops})=[n]$ and the loop assignment $\loops=[m]$ requires the nSSP.  
\end{theorem}
\begin{proof}
Let $G=\dvec{P}_{n,\loops}$, $A \in \mathcal{M}(G)$  and $X=(x_{ij}) \in \overline{\mathcal{M}}(G^c)$ be arbitrary, such that $[A,X\trans]=0$. By Lemma~\ref{lem:induced-path-loopless} we have $x_{i,j}=0$ for all $m\leq i, j\leq n$ such that $|i-j|=2s+1$, where $1\leq 2s+1 \leq n-m$ and particularly \begin{equation}\label{eq:odd-distance}
    x_{m,m+2s+1}=0,
\end{equation}
so $X \in \overline{\mathcal{M}}(G_1^c)$ for $G_1:=G+\{\{i,j\}:m\leq i,j\leq n, |i-j|=2s+1, 1\leq 2s+1 \leq n-m\}.$
By Lemma~\ref{lem:pendent-looped-path-on-a-path} it follows that $x_{ij}=0$ for all $i+j\leq 2m+1$, or equivalently $X\in \overline{\mathcal{M}}(G_{2}^c)$ for $G_{2}:=G_1+\{\{j,m+k\} \colon k\leq \min\{m+1,n-m\},j\in [m-k+1]\}.$

Define
\begin{align*}
G_{t}&:=G_{1}+\{\{i,j\}:i+j\leq 2m+t-1\}
\end{align*}
for $t$, $3\leq t\leq n-m+1$.
Note that if $X\in \overline{\mathcal{M}}(G_{t}^c)$, then for all $i$ and $j$ such that $i+j=2m+t-1,$ the equation \begin{equation}\label{eq:iplusj}[A,X\trans]_{i,j}=a_{i,i+1}x_{j,i+1}-a_{j+1,j}x_{j+1,i}=0\end{equation} holds, because $N_G^+[i]\cap N_{G_{t}}^+[j]^c=\{i+1\}$ and $N^-_G[j]\cap N^-_{G_{t}}[i]^c=\{j+1\}$. Notice that variable $x_{j+1,i}$ appears in two consecutive equations $[A,X\trans]_{i,j}=0$ and $[A,X\trans]_{i-1,j+1}=0$.

Now we will prove the following implications inductively:
\begin{equation}
    \label{eq:implications}
X \in \overline{\mathcal{M}}(G_{2s}^c) \implies X \in 
\overline{\mathcal{M}}(G_{2s+1}^c) \implies X\in \overline{\mathcal{M}}(G_{2s+2}^c)
\end{equation}
for $1\leq s \leq \frac{1}{2}(n-m-1)$.

Let $s=1$.
By (\ref{eq:odd-distance}) we have $x_{m,m+3}=0$, and so from $$0=[A,X\trans]_{m+2,m}=-a_{m,m}x_{m,m+2}+a_{m+2,m+3}x_{m,m+3}$$
  it follows $x_{m,m+2}=0.$  Since variable $x_{m,m+2}=0$ appears by~\eqref{eq:iplusj} in the equation $[A,X\trans]_{m+2,m-1}=0$ as well, we conclude $x_{m-1,m+3}=0$ and consecutively $x_{i,j}=0$ for all $i, j$ such that $i+j=2m+2$. So, $X\in \overline {\mathcal{M}}(G_{3}^c)$. In the same fashion, while observing the equations~\eqref{eq:iplusj} for $i+j=2m+2$, it can be proven that $X \in \overline {\mathcal{M}}(G_{4}^c)$.

    Suppose now $X \in \overline{\mathcal{M}}(G_{2s}^c)$ for some $s$, $1\leq s \leq \frac{1}{2}(n-m-1)$.  By~\eqref{eq:odd-distance} and~\eqref{eq:iplusj} it follows that  $$0=[A,X\trans]_{m+2s,m}=-a_{m,m}x_{m,m+2s}+a_{m+2s,m+2s+1}x_{m,m+2s+1}=-a_{m,m}x_{m,m+2s},$$ and so $x_{m,m+2s}=0$.
 For $i, j$, such that $i+j=2m+2s-1$ the equation (\ref{eq:iplusj}) holds.
 Since variable $x_{m,m+2s}=0$ appears in the equation $[A,X\trans]_{m+2s,m-1}=0$, we  conclude recursively as before that $x_{i,j}=0$ for all $i, j$ such that $i+j=2m+2s$. So, $X \in \overline{\mathcal{M}}(G_{2s+1}^c)$.

 To prove the second implication it is enough to observe the equations (\ref{eq:iplusj}) for $i+j=2m+2s$. Equation $[A,X\trans]_{m+2s+1,m-1}=a_{m+2s+1,m+2s+2}x_{m-1,m+2s+2}=0$ will imply $x_{m-1,m+2s+2}=0.$ As before, two consecutive equations $[A,X\trans]_{i,j}=0$ for $i+j=2m+2s$ have a variable in common, so $x_{m-1,m+2s+2}=0$ will force $x_{i,j}=0$ for all $i+j=2m+2s+1,$ so $X\in\overline{\mathcal{M}}(G_{2s+2}^c).$
 
Observe that $X\in\overline{\mathcal{M}}(G_{n-m+1}^c)$. If  $2s+1>n-m$, then  $0=[A,X\trans]_{n,2m+2s-1-n}=-a_{m+1,m}x_{m+1,n}$. As above, observing the consecutive equations $[A,X\trans]_{i,j}=0$ and $[A,X\trans]_{i-1,j+1}=0$, for $i+j=2m+2s-1$, will imply $x_{i,j}=0$ for all $i+j=2m+2s$. Also, the equations (\ref{eq:iplusj}) for $i+j=2m+2s+1$ will imply $x_{i,j}=0$ for $i+j=2m+2s+2.$ Repeating the same process for all $s>1$ such that $2s+1>n-m$ will force all the other entries in $X$ to be zero, so $X=O$ and $G$ requires the nSSP.
\end{proof}

Now it follows that in the special case when $G$ in Lemma~\ref{lem:noloops-all-loops-nssp}(\ref{all-loops}) is a double path with all the loops, it does not only allow but it requires the nSSP:

\begin{corollary}
   For every $n$, the double path $P_{n,[n]}$ requires the nSSP. 
\end{corollary}

Note that the number of loops of a double path $\dvec{P}_{n,\loops}$ does not reassure that it would require the nSSP. In Theorem~\ref{thm:double_paths-first-m-loops} we showed that for any $k\in [n]$ there exists a double path with exactly $k$ loops that requires the nSSP and in Theorem~\ref{thm:double_paths-lots-of-loops} we showed that a certain number and position of loops guarantee the nSSP as well. In the next example we show there exist double paths with one loop that require, not require or not even allow the nSSP. Also, we construct a family of double paths with all but one loop that does not require the nSSP.

\begin{theorem}\label{thm:double-path-vtx-2}
Let $n\in \bN$ and let $\dvec{P}_{n,\loops}$ be a double path with $V(\dvec{P}_{n,\loops})=[n]$.
\begin{enumerate}[(A)]
\item \label{1:loops=1n+and+c}If $\loops=\{1,n\}$, then $\dvec{P}_{n,\loops}$ does not require, but allows the nSSP, and
$\dvec{P}_{n,\loops^c}$ does not require the nSSP.

\item \label{2:loops=2,even-case}If  $n$ is even, then $\dvec{P}_{n,\{2\}}$ requires the nSSP. 
 \item\label{3:n-odd-cases} Let $n\geq 3$ be an odd integer.
 \begin{enumerate}
      \item\label{item:middle-loop} If $\loops=\{\frac{n+1}{2}\}$,   
      then $\dvec{P}_{n,\loops}$ does not require the nSSP.
      \item\label{item:without-middle-loop} If $\loops=[n]\setminus\{\frac{n+1}{2}\}$, then $\dvec{P}_{n,\loops}$ does not require, but allows the nSSP. 
   \item\label{item:even-loops} If $\loops\subseteq \{2i \colon i\in \left[\frac{n-1}{2}\right]\}$, then $\dvec{P}_{n,\loops}$ does not allow the nSSP.
   
\end{enumerate}
 
\end{enumerate}

\end{theorem}

\begin{proof}
\begin{enumerate}[(A)]
\item  It is straightforward to check that the adjacency matrix $A$ of $\dvec{P}_{n,\{1,n\}}$ and the matrix $X=E-A$, where $E$ is the matrix of all ones, commute. Since they are both symmetric and $A\circ X=O$, it follows that $A$ does not have the nSSP. By Theorem~\ref{thm:double_paths-first-m-loops} and Lemma~\ref{lem:superpattern-lemma} it follows that  $\dvec{P}_{n,\{1,n\}}$ allows, but does not require the nSSP.

Note that $A$ and $X$ are both matrices in $\{0,1\}^{n\times n}$, and so for $B:=I-A\in  {\mathcal M}(\dvec{P}_{n,\{1,n\}^c})$ and $Y:=I-X$ we have $B\circ Y=[B,Y]=O$. Therefore, $(\dvec{P}_{n,\{1,n\}^c})$ does not require the nSSP.

\item  Let $n$ be even. 
 Let $A \in \mathcal{M}(G)$ and $X \in \overline{\mathcal{M}}(G^c)$ such that $[A,X\trans]=O$. By Lemma~\ref{lem:induced-path-loopless} on the induced subgraph on vertices $\{2,\ldots,n\}$ we conclude that $x_{ij}=0$ for all $2\leq i, j \leq n$ such that $|i-j|$ is odd. Let us denote $G_1=G+\{\{i,j\} : 2\leq i, j \leq n, |i-j| \ \text{odd}\}$.

Now we prove that $X \in \overline{\mathcal{M}}(G_2^c)$ where $G_2=G_1+\{\{1,k\} : k \ \text {even}\}$.   First observe that for all odd $k$ we have \begin{equation}\label{eq:1st-column}
0=[A,X\trans]_{k,1}=a_{k,k-1}x_{1,k-1}+a_{k,k+1}x_{1,k+1}\end{equation}  since $N_G^+[k]\cap N_{G_1}^+[1]^c=\{k-1, k+1\}$ and $N_G^-[1]\cap N_{G_1}^-[k]^c=\emptyset$. First, for $k=3$ we use that $x_{12}=0$ and  $N_G^+[3]\cap N_{G_1}^+[1]^c=\{4\}$ and $N_G^-[1]\cap N_{G_1}^-[3]^c=\emptyset$, therefore  $x_{14}=x_{41}=0$ by Lemma~\ref{lem:rule1} and Remark~\ref{rem:symmetric-pattern}.   For all odd $k\geq 3$,  we use inductively that $x_{1,k-1}=x_{k-1,1}=0$ and therefore by~\eqref{eq:1st-column} we have $x_{1,k+1}=x_{k+1,1}$. Hence $X \in \overline{\mathcal{M}}(G_2^c)$.
   
   Now, for odd $k\geq 3$ we have $N_G^+[k+1]\cap N_{G_2}^+[2]^c=\emptyset$ and $N_G^-[2]\cap N_{G_2}^-[k+1]^c=\{2\}$, and so it follows that $x_{2,k+1}=x_{k+1,2}=0$ by Lemma~\ref{lem:rule1} and Remark~\ref{rem:symmetric-pattern}. This implies $X \in \overline{\mathcal{M}}(G_3^c)$ where $ G_3=G_2+\{\{2,k\}: k \ \text{even}\}$. 
   
   Similarly as in~\eqref{eq:1st-column} we now observe that $0=[A,X\trans]_{k,1}=a_{k,k-1}x_{1,k-1}+a_{k,k+1}x_{1,k+1}$ also for all even $k$. Since $n$ is even, let $k=n$ and since $a_{k,k+1}=0$, it follows that $x_{1,n-1}=0$. Inductively for $k<n$ it follows that $x_{1,k-1}=0$ for all even $k$. By symmetry, $X \in \overline{\mathcal{M}}(G_4^c)$ where $G_4=G_3+\{\{1,k\} : k \ \text {odd}\}$. 

   At this point $x_{1k}=x_{2k}=x_{k1}=x_{k2}=0$ for all $k\in [n]$. So, $$A=\begin{pmatrix} 
0 & \alpha \, {\bf e}_1\trans \\
\beta\, {\bf e}_1 & A_1
\end{pmatrix} \in {\mathcal M}(G), \ X=\begin{pmatrix}
   0 & \textbf{0}  \\
\textbf{0}\trans & X_1
\end{pmatrix}\in \overline{{\mathcal M}}(G^c),$$
where $\alpha, \beta\in \bR\setminus \{0\}$, $A_1=(a_{ij})\in \mathcal{M}(H)$ and $X_1 \in \overline{\mathcal{M}}(H^c) $ for $H=\dvec{P}_{n-1,\{1\}}$.
\\
Since $O=[A,X\trans]=\begin{pmatrix}
    0 & - \alpha {\bf e}_1\trans X_1\trans \\
    -\beta X_1\trans {\bf e}_1 &[A_1,X_1\trans]
\end{pmatrix}$,
and since $H$ requires the nSSP by the Theorem~\ref{thm:double_paths-first-m-loops}, it follows that $X_1=0$ and so  $X=0$.

\item  Assume $n$ is odd.

(a) Let $k=\frac{n+1}{2}$ and let $A$ be the adjacency matrix of the graph. Then 
\begin{equation}\label{eq:AX-middle-looped}A=\begin{pmatrix}
        B&{\bf e}_{k}&O\\
        {\bf e}_{k}\trans&1&{\bf e}_{1}\trans\\
        O&{\bf e}_{1}&B
    \end{pmatrix} \text{ and } X:=\begin{pmatrix}
        -I & {\bf 0} & Y\\
        {\bf 0} & 0 & {\bf 0}\\
        Y & {\bf 0} & -I
    \end{pmatrix},
    \end{equation}
    where $B={\rm Adj}(P_k)$ is an adjacency matrix of a loopless path, and $Y=(y_{ij})\in \bR^{k\times k}$ is defined as $$y_{ij}=\begin{cases}
        1, & i+j=k+1,\\
        0, & \text{otherwise}.
    \end{cases}$$
    It is easy to see that matrices $A$ and $X\trans=X$ commute and $A\circ X=O$. Hence $A$ does not have the nSSP.

    (b) 
 Let $\loops=[n]\setminus\{\frac{n+1}{2}\}$ and $A$ and $X$ as in~\eqref{eq:AX-middle-looped}. Observe that $B:=I-A\in \calM(\dvec{P}_{n,\loops})$ and $Y=I-X\in\overline{\calM}(\dvec{P}_{n,\loops}^c)$. Since $B\circ Y=O$ and $[B,Y]=[A,X]=O$, it follows that $B$ does not have the nSSP. In Theorem~\ref{thm:double_paths-first-m-loops} we have shown that $\dvec{P}_{n,[1]}$ requires the nSSP and so $\dvec{P}_{n,\loops}$ allows the nSSP by Lemma~\ref{lem:superpattern-lemma}, but does not require it.

 (c) For $n=3$ it is trivial to check that $A=\left(
\begin{array}{ccc}
 0 & a_{1,2} & 0 \\
 a_{2,1} & a_{2,2} & a_{2,3} \\
 0 & a_{3,2} & 0 \\
\end{array}
\right)$ and $X=\left(
\begin{array}{ccc}
 a_{2,3} a_{3,2} & 0 & -a_{2,1}a_{3,2} \\
 0 & 0 & 0 \\
 -a_{1,2}a_{2,3} & 0 & a_{1,2} a_{2,1} \\
\end{array}
\right)$ satisfy $[A,X\trans]=0$.
Let $G_1=\dvec{P}_{n,\loops}, G=\dvec{P}_{n+2,{\mathcal K}},$ where $\loops\subseteq {\mathcal K}\subseteq \loops\cup\{(n+1,n+1)\}$ and take arbitrary $A \in \mathcal{M}(G)$.
We will inductively show that for every matrix
$$A=\begin{pmatrix} 
A_1 & \alpha \, {\bf e}_n & {\bf 0}\\
\beta\, {\bf e}_n\trans &  \epsilon& \gamma\\
 {\bf 0}\trans & \delta & 0
\end{pmatrix} \in {\mathcal M}(G)$$
where $A_1=(a_{ij})\in \mathcal{M}(G_1)$ and $\alpha, \beta,\gamma,\delta \in \bR\setminus \{0\}$, $\epsilon \in \mathbb{R}$ there exists a matrix $$X=\begin{pmatrix}
    X_1&\bf{0}&{\bf w}\\
    {\bf 0} & 0 & 0\\
    {\bf v}\trans & 0 & \psi\\
\end{pmatrix}\in \overline{{\mathcal M}}(G^c),$$ such that $[A,X\trans]=O$ and so that all even-indexed rows and all even-indexed columns of $X_1$ are equal to ${\bf 0}$.

Suppose the statement is true for some odd $n\geq 3$ and that $A$ and $X$ are as above. By inductive hypothesis, $A_1$ does not have the nSSP, therefore there exists $X_1=(x_{ij})\in \overline{{\mathcal M}}(G_1^c)$, such that $[A_1,X_1\trans]=O$ and $x_{ij}=0$ if $i$ is even or if $j$ is even. Therefore for all $i\in [\frac{n-1}{2}]$ and all $j\in [\frac{n+1}{2}]$ we have
\begin{align*}
0&=[A_1,X_1\trans]_{2i,2j-1}= a_{2i,2i-1}x_{2j-1,2i-1}+a_{2i,2i+1}x_{2j-1,2i+1} \; \text{ and }\\
0&=[A_1,X_1\trans]_{2j-1,2i}= -a_{2i-1,2i}x_{2i-1,2j-1}-a_{2i+1,2i}x_{2i+1,2j-1}.
\end{align*}
Let us define  $\psi=\frac{\alpha \beta}{\gamma \delta}x_{nn}$ and
$${\bf w}_i:=\begin{cases}
    -\frac{\beta}{\gamma}x_{in}, & i \text{ odd},\\
    0, & i \text{ even},\\
\end{cases} \quad \text{ and }\quad
{\bf v}_i:=\begin{cases}
    -\frac{\alpha}{\delta}x_{ni}, & i  \text{ odd},\\
    0, & i \text{ even}.\\
\end{cases}$$
Now, it is a straightforward computation that 
$$[A,X\trans]=\begin{pmatrix}
    [A_1,X_1\trans] & - \alpha X_1\trans {\bf e}_n-\delta {\bf v} & A_1 {\bf v}\\
    \beta {\bf e}_n\trans X_1\trans +\gamma{\bf w}\trans &0&\beta {\bf e}_n\trans {\bf v}+\gamma \psi\\
    -{\bf w}\trans A_1 & -\delta\psi -\alpha {\bf w}\trans {\bf e}_n   &0
\end{pmatrix}=O,$$
and so we proved that $ \dvec{P}_{n,\loops}$ does not allow the nSSP for all odd $n$.\qedhere
\end{enumerate}
\end{proof}

We finish the paper with several examples that illustrate the difficulty of Problem~\ref{problem-6.5}.

\begin{example}
  Let $\loops=\{1,n\}^c$. Note that for an even $n$, Theorem~\ref{thm:double-path-vtx-2}(\ref{1:loops=1n+and+c}) and (\ref{2:loops=2,even-case}) and Lemma~\ref{lem:superpattern-lemma} imply that $\dvec{P}_{n,\loops}$ allows, but does not require the nSSP.
  
  However, if $n$ is odd, then $\dvec{P}_{3,\loops}$ does not allow the nSSP by Theorem~\ref{thm:double-path-vtx-2}(\ref{item:even-loops}). We will show in Example~\ref{ex:central_being_looped} that $\dvec{P}_{5,\{3\}}$ allows the nSSP, therefore by Lemma~\ref{lem:superpattern-lemma} and Theorem~\ref{thm:double-path-vtx-2}(\ref{1:loops=1n+and+c}) it follows that $\dvec{P}_{5,\loops}$ allows but does not require the nSSP. A similar argument shows that the same is true for $\dvec{P}_{7,\loops}$, since it can be shown that $$A=\left(
\begin{array}{ccccccc}
 0 & 1 & 0 & 0 & 0 & 0 & 0 \\
 1 & 1 & 1 & 0 & 0 & 0 & 0 \\
 0 & 1 & 1 & 1 & 0 & 0 & 0 \\
 0 & 0 & 2 & 2 & 2 & 0 & 0 \\
 0 & 0 & 0 & 1 & 1 & 1 & 0 \\
 0 & 0 & 0 & 0 & 1 & 1 & 1 \\
 0 & 0 & 0 & 0 & 0 & 2 & 0 \\
\end{array}
\right)$$ has the nSSP.
\end{example}

The next two examples consider cases when $|\loops|=1$.

\begin{example}
    Note that Theorem~\ref{thm:double-path-vtx-2}(\ref{2:loops=2,even-case}) is not true even for an even $n$ if we move the loop from vertex $2$ to another position. Namely, if $$A=
    \left(
\begin{array}{cccccc}
 0 & 1 & 0 & 0 & 0 & 0 \\
 1 & 0 & -1 & 0 & 0 & 0 \\
 0 & 1 & 1 & 1 & 0 & 0 \\
 0 & 0 & 1 & 0 & 1 & 0 \\
 0 & 0 & 0 & -1 & 0 & 1 \\
 0 & 0 & 0 & 0 & 2 & 0 \\
\end{array}
\right)
    \text{ and }
    X=\left(\begin{array}{cccccc}
 -1 & 0 & 1 & -1 & 2 & -2 \\
 0 & -2 & 0 & 3 & -1 & 4 \\
 -1 & 0 & 0 & 0 & -1 & 0 \\
 1 & -3 & 0 & 2 & 0 & 2 \\
 2 & -1 & 1 & 0 & -1 & 0 \\
 -1 & 2 & 0 & -1 & 0 & -2 \\
\end{array}
\right),$$
    it is easy to see that $A\circ X=[A,X\trans]=O$ and therefore $A$ does not have the nSSP. It is straightforward to verify that the adjacency matrix of $\dvec{P}_{6,\{3\}}$ has the nSSP. Consequently, $\dvec{P}_{6,\{3\}}$ allows the nSSP, but does not require it.
\end{example}

\begin{example}\label{ex:central_being_looped}
Observe that for an odd $n\geq 3$ the double path $\dvec{P}_{n,\{\frac{n+1}{2}\}}$ does not require the nSSP by Theorem~\ref{thm:double-path-vtx-2}(\ref{item:middle-loop}) and in the case $n=-1\mod 4$, then $\frac{n+1}{2}$ is an even number and so by Theorem~\ref{thm:double-path-vtx-2}(\ref{item:even-loops}) it does not even allow the nSSP. 

However, in the case $n=1\mod 4$, $\dvec{P}_{n,\{\frac{n+1}{2}\}}$ can allow the nSSP.  For example, for the matrix
$$A=\begin{pmatrix}
 0 & 1 & 0 & 0 & 0 \\
 1 & 0 & -1 & 0 & 0 \\
 0 & 1 & 1 & 1 & 0 \\
 0 & 0 & 1 & 0 & -1 \\
 0 & 0 & 0 & 1 & 0 \\
\end{pmatrix}\in {\mathcal M}(\dvec{P}_{5,\{3\}})$$
let $X\in \overline{\mathcal M}(\dvec{P}_{5,\{3\}}^c)$ satisfy $[A, X\trans]=O$. Then the corresponding homogeneous linear system of equations with variables in nonzero entries of $X$ in the order $$\{x_{11}, x_{13}, x_{14}, x_{15}, x_{22}, x_{24}, x_{25}, x_{31}, x_{35}, x_{41}, x_{42}, x_{44}, x_{51}, x_{52}, x_{53}, x_{55}\}$$ has the matrix
$${\scriptsize \left(
\begin{array}{cccccccccccccccc}
 0 & 1 & 1 & 0 & 0 & 0 & 0 & 0 & 0 & 0 & 0 & 0 & 0 & 0 & 0 & 0 \\
 1 & -1 & 0 & 0 & -1 & 0 & 0 & 0 & 0 & 0 & 0 & 0 & 0 & 0 & 0 & 0 \\
 0 & 1 & 0 & -1 & 0 & -1 & 0 & 0 & 0 & 0 & 0 & 0 & 0 & 0 & 0 & 0 \\
 0 & -1 & 0 & 0 & 1 & 1 & 0 & 0 & 0 & 0 & 0 & 0 & 0 & 0 & 0 & 0 \\
 0 & 0 & -1 & 0 & 0 & 0 & -1 & 0 & 0 & 0 & 0 & 0 & 0 & 0 & 0 & 0 \\
 0 & 0 & 1 & 0 & 0 & 0 & -1 & 0 & 0 & 0 & 0 & 0 & 0 & 0 & 0 & 0 \\
 -1 & 0 & 0 & 0 & 1 & 0 & 0 & -1 & 0 & 0 & 0 & 0 & 0 & 0 & 0 & 0 \\
 0 & 0 & 0 & -1 & 0 & 1 & 0 & 0 & -1 & 0 & 0 & 0 & 0 & 0 & 0 & 0 \\
 0 & 0 & 0 & 0 & 0 & 0 & 1 & 0 & -1 & 0 & 0 & 0 & 0 & 0 & 0 & 0 \\
 0 & 0 & 0 & 0 & 0 & 0 & 0 & -1 & 0 & -1 & 0 & 0 & 0 & 0 & 0 & 0 \\
 0 & 0 & 0 & 0 & 1 & 0 & 0 & 1 & 0 & 0 & -1 & 0 & 0 & 0 & 0 & 0 \\
 0 & 0 & 0 & 0 & 0 & 1 & 0 & 0 & -1 & 0 & 0 & -1 & 0 & 0 & 0 & 0 \\
 0 & 0 & 0 & 0 & 0 & 0 & 0 & -1 & 0 & 0 & 1 & 0 & -1 & 0 & 0 & 0 \\
 0 & 0 & 0 & 0 & 0 & 0 & 0 & 0 & 0 & 1 & 0 & 0 & 0 & -1 & 0 & 0 \\
 0 & 0 & 0 & 0 & 0 & 0 & 0 & 0 & 0 & 1 & 0 & 0 & 0 & 1 & 0 & 0 \\
 0 & 0 & 0 & 0 & 0 & 0 & 0 & 0 & 0 & 0 & 1 & 1 & 0 & 0 & -1 & 0 \\
 0 & 0 & 0 & 0 & 0 & 0 & 0 & 0 & 0 & 0 & 1 & 0 & 1 & 0 & -1 & 0 \\
 0 & 0 & 0 & 0 & 0 & 0 & 0 & 0 & 0 & 0 & 0 & 0 & 0 & 1 & 1 & 0 \\
 0 & 0 & 0 & 0 & 0 & 0 & 0 & 0 & -1 & 0 & 0 & 1 & 0 & 0 & 0 & -1 \\
 0 & 0 & 0 & 0 & 0 & 0 & 0 & 0 & 0 & 0 & 0 & 1 & 0 & 0 & 1 & -1 \\
\end{array}
\right)},$$
which is also called a verification matrix of $A$. Observe that it has a full column rank, which implies $X=O$. This proves that $A$ has the nSSP and so $\dvec{P}_{5,\{3\}}$ allows, but does not require the nSSP.
\end{example}

These examples illustrate that resolving the subproblem of Problem~\ref{problem-6.5} for the case where $|\loops|=1$ is a non-trivial and compelling task.
Furthermore, the preceding examples naturally motivate the following question.

\begin{question}
Is it true that if $n$ is even and $\loops\ne \emptyset$, then $\dvec{P}_{n,\loops}$ allows the nSSP?
\end{question}

We finish the paper by resolving Problem~\ref{problem-6.5} for $n\leq 5$.

\begin{example}\label{ex:n<=5}
  Note that the results in Section~\ref{sec:double-paths} give a complete answer to Problem~\ref{problem-6.5} for graphs on at most four vertices. We gather their properties in Table~\ref{tab:small-paths}, where we present all nonisomorphic double paths.

In the case $n=5$ there are six loop assignments $\loops \subseteq [5]$ that are not resolved by the Theorems~\ref{thm:double_paths-lots-of-loops}, \ref{thm:double_paths-first-m-loops}, \ref{thm:double-path-vtx-2} and Example~\ref{ex:central_being_looped}.

 Let us show that for $\loops=\{1,2,4\}$ the double path $\dvec{P}_{5,\loops}$ requires the nSSP.
Take arbitrary $A=(a_{ij})\in \mathcal{M}(\dvec{P}_{5,\loops})$ and $X=(x_{ij})\in \overline{\mathcal{M}}(\dvec{P}_{5,\loops}^c)$.  Using Lemma~\ref{lem:rule1} eight times there exists a  sequence of supergraphs of $G=\dvec{P}_{5,\loops}$ ending in $K_5^\circ$. Namely, on Figure~\ref{fig:P_5-loops=124-requires} we show a sequence, where on each arrow between $G_{\ell}$ and $G_{\ell+1}$ we write the triple $(i,j,k)$, so that $(i,j,k)$ satisfies condition~\eqref{lem:rule1out} and $(j,i,k)$ satisfies~\eqref{rule1in}, and so $G_{\ell+1}=G_{\ell}+\{\{j,k\}\}$. By Theorem~\ref{thm:rule1-nSSP} it follows that $\dvec{P}_{5,\{1,2,4\}}$ requires the nSSP.

\begin{figure}[htb]
\begin{tikzpicture}[scale=0.5]
\foreach \i in {1,...,4} {
 \draw[thick] (\i,0)--(\i+1,0);
 \node[draw=none] at (\i,-0.7) {\small{\i}};
 \node[fill=white] (\i) at (\i,0) {};
};
 \node[fill=white] (5) at (5,0) {};
\node[draw=none] at (5,-0.7) {\small{5}};
\tikzset{every loop/.style={min distance=4mm,in=120,out=60,looseness=30}}
\foreach \i in {1,2,4} {
     \draw[thick,->] (\i) to[loop above] (\i);
};
\node[rectangle, draw=none] at (3,-2) {$G_0=\dvec{P}_{5,\{1,2,4\}}$};

\begin{scope}[shift={(9,0)}]
\node[draw=none] at (-1.5,1) {\small $(2,1,3)$};
\draw[->] (-2,0)--(-1,0);
\foreach \i in {1,...,4} {
 \draw[thick] (\i,0)--(\i+1,0);
 \node[draw=none] at (\i,-0.7) {\small{\i}};
 \node[fill=white] (\i) at (\i,0) {};
};
 \node[fill=white] (5) at (5,0) {};
\node[draw=none] at (5,-0.7) {\small{5}};
 \draw[thick,green] (1) to[bend right] (3);
\tikzset{every loop/.style={min distance=4mm,in=120,out=60,looseness=30}}
\foreach \i in {1,2,4} {
     \draw[thick,->] (\i) to[loop above] (\i);
};
\node[rectangle, draw=none] at (3,-2) {$G_1$};
\end{scope}

\begin{scope}[shift={(18,0)}]
\node[draw=none] at (-1.5,1) {\small $(3,1,4)$};
\draw[->] (-2,0)--(-1,0);
\foreach \i in {1,...,4} {
 \draw[thick] (\i,0)--(\i+1,0);
 \node[draw=none] at (\i,-0.7) {\small{\i}};
 \node[fill=white] (\i) at (\i,0) {};
};
 \node[fill=white] (5) at (5,0) {};
\node[draw=none] at (5,-0.7) {\small{5}};
 \draw[thick] (1) to[bend right] (3);
 \draw[thick,green] (1) to[bend right] (4);
\tikzset{every loop/.style={min distance=4mm,in=120,out=60,looseness=30}}
\foreach \i in {1,2,4} {
     \draw[thick,->] (\i) to[loop above] (\i);
};
\node[rectangle, draw=none] at (3,-2) {$G_2$};
\end{scope}

\begin{scope}[shift={(0,-5)}]
\node[draw=none] at (-1.5,1) {\small $(3,5,2)$};
\draw[->] (-2,0)--(-1,0);
\foreach \i in {1,...,4} {
 \draw[thick] (\i,0)--(\i+1,0);
 \node[draw=none] at (\i,-0.7) {\small{\i}};
 \node[fill=white] (\i) at (\i,0) {};
};
 \node[fill=white] (5) at (5,0) {};
\node[draw=none] at (5,-0.7) {\small{5}};
 \draw[thick] (1) to[bend right] (3);
 \draw[thick] (1) to[bend right] (4);
 \draw[thick,green] (2) to[bend right] (5);
\tikzset{every loop/.style={min distance=4mm,in=120,out=60,looseness=30}}
\foreach \i in {1,2,4} {
     \draw[thick,->] (\i) to[loop above] (\i);
};
\node[rectangle, draw=none] at (3,-2) {$G_3$};
\end{scope}

\begin{scope}[shift={(9,-5)}]
\node[draw=none] at (-1.5,1) {\small $(1,5,1)$};
\draw[->] (-2,0)--(-1,0);
\foreach \i in {1,...,4} {
 \draw[thick] (\i,0)--(\i+1,0);
 \node[draw=none] at (\i,-0.7) {\small{\i}};
 \node[fill=white] (\i) at (\i,0) {};
};
 \node[fill=white] (5) at (5,0) {};
\node[draw=none] at (5,-0.7) {\small{5}};
 \draw[thick] (1) to[bend right] (3);
 \draw[thick] (1) to[bend right] (4);
 \draw[thick] (2) to[bend right] (5);
 \draw[thick,green] (1) to[bend right=45] (5);
\tikzset{every loop/.style={min distance=4mm,in=120,out=60,looseness=30}}
\foreach \i in {1,2,4} {
     \draw[thick,->] (\i) to[loop above] (\i);
};
\node[rectangle, draw=none] at (3,-2) {$G_4$};
\end{scope}

\begin{scope}[shift={(18,-5)}]
\node[draw=none] at (-1.5,1) {\small $(1,4,2)$};
\draw[->] (-2,0)--(-1,0);
\foreach \i in {1,...,4} {
 \draw[thick] (\i,0)--(\i+1,0);
 \node[draw=none] at (\i,-0.7) {\small{\i}};
 \node[fill=white] (\i) at (\i,0) {};
};
 \node[fill=white] (5) at (5,0) {};
\node[draw=none] at (5,-0.7) {\small{5}};
 \draw[thick] (1) to[bend right] (3);
 \draw[thick] (1) to[bend right=45] (4);
 \draw[thick] (2) to[bend right=45] (5);
 \draw[thick] (1) to[bend right=50] (5);
 \draw[thick,green] (2) to[bend right] (4);
\tikzset{every loop/.style={min distance=4mm,in=120,out=60,looseness=30}}
\foreach \i in {1,2,4} {
     \draw[thick,->] (\i) to[loop above] (\i);
};
\node[rectangle, draw=none] at (3,-2) {$G_5$};
\end{scope}

\begin{scope}[shift={(0,-10)}]
\node[draw=none] at (-1.5,1) {\small $(2,5,3)$};
\draw[->] (-2,0)--(-1,0);
\foreach \i in {1,...,4} {
 \draw[thick] (\i,0)--(\i+1,0);
 \node[draw=none] at (\i,-0.7) {\small{\i}};
 \node[fill=white] (\i) at (\i,0) {};
};
 \node[fill=white] (5) at (5,0) {};
\node[draw=none] at (5,-0.7) {\small{5}};
 \draw[thick] (1) to[bend right] (3);
 \draw[thick] (1) to[bend right=45] (4);
 \draw[thick] (2) to[bend right=45] (5);
 \draw[thick] (1) to[bend right=50] (5);
 \draw[thick] (2) to[bend right] (4);
 \draw[thick,green] (3) to[bend right] (5);
\tikzset{every loop/.style={min distance=4mm,in=120,out=60,looseness=30}}
\foreach \i in {1,2,4} {
     \draw[thick,->] (\i) to[loop above] (\i);
};
\node[rectangle, draw=none] at (3,-2) {$G_6$};
\end{scope}

\begin{scope}[shift={(9,-10)}]
\node[draw=none] at (-1.5,1) {\small $(2,3,3)$};
\node[draw=none] at (-1.5,-1) {\small $(4,5,5)$};
\draw[->] (-2,0)--(-1,0);
\foreach \i in {1,...,4} {
 \draw[thick] (\i,0)--(\i+1,0);
 \node[draw=none] at (\i,-0.7) {\small{\i}};
 \node[fill=white] (\i) at (\i,0) {};
};
 \node[fill=white] (5) at (5,0) {};
\node[draw=none] at (5,-0.7) {\small{5}};
 \draw[thick] (1) to[bend right] (3);
 \draw[thick] (1) to[bend right=45] (4);
 \draw[thick] (2) to[bend right=45] (5);
 \draw[thick] (1) to[bend right=50] (5);
 \draw[thick] (2) to[bend right] (4);
 \draw[thick] (3) to[bend right] (5);
\tikzset{every loop/.style={min distance=4mm,in=120,out=60,looseness=30}}
\foreach \i in {1,2,4} {
     \draw[thick,->] (\i) to[loop above] (\i);
};
\foreach \i in {3} {
     \draw[thick,->,green] (\i) to[loop above] (\i);
};
\node[rectangle, draw=none] at (3,-2) {$G_7$};
\end{scope}

\begin{scope}[shift={(18,-10)}]
\node[draw=none] at (-1.5,1) {\small $(4,5,5)$};
\draw[->] (-2,0)--(-1,0);
\foreach \i in {1,...,4} {
 \draw[thick] (\i,0)--(\i+1,0);
 \node[draw=none] at (\i,-0.7) {\small{\i}};
 \node[fill=white] (\i) at (\i,0) {};
};
 \node[fill=white] (5) at (5,0) {};
\node[draw=none] at (5,-0.7) {\small{5}};
 \draw[thick] (1) to[bend right] (3);
 \draw[thick] (1) to[bend right=45] (4);
 \draw[thick] (2) to[bend right=45] (5);
 \draw[thick] (1) to[bend right=50] (5);
 \draw[thick] (2) to[bend right] (4);
 \draw[thick] (3) to[bend right] (5);
\tikzset{every loop/.style={min distance=4mm,in=120,out=60,looseness=30}}
\foreach \i in {1,2,3,4} {
     \draw[thick,->] (\i) to[loop above] (\i);
};
\foreach \i in {5} {
     \draw[thick,->,green] (\i) to[loop above] (\i);
};
\node[rectangle, draw=none] at (3,-2) {$G_8=K_5^\circ$};
\end{scope}
\end{tikzpicture}
\caption{Sequence $\dvec{P}_{5,\{1,2,4\}}\subset G_1 \subset G_2 \subset \ldots \subset G_8=K_5^{\circ}$ which shows that $\dvec{P}_{5,\{1,2,4\}}$ requires the nSSP. Every undirected edge represents a double arc and $G_{\ell+1}=G_{\ell}+\{\{j,k\}\}$, where the double arc $\{j,k\}$ is coloured green.}\label{fig:P_5-loops=124-requires}
\end{figure}
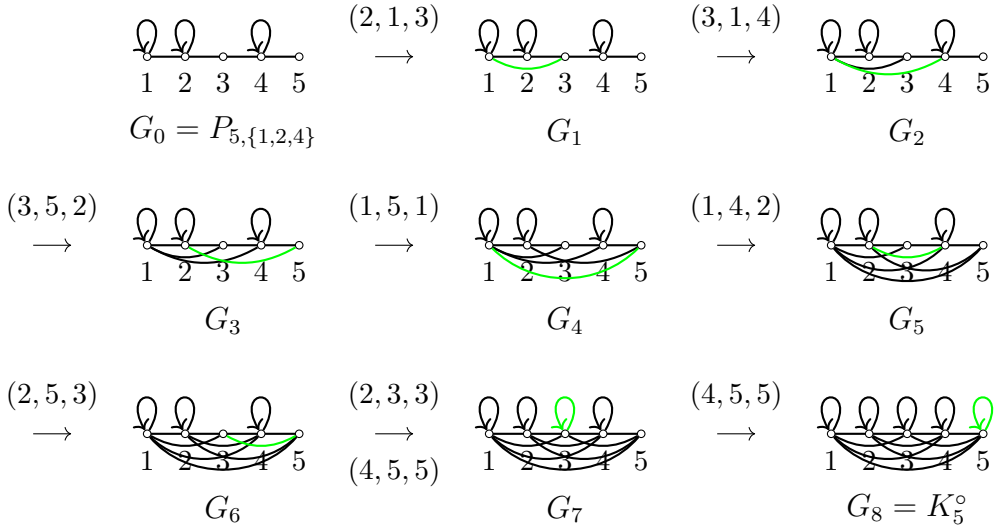

\begin{table}[htb]
    \centering
    \begin{tabular}{c||c|c}
       $\loops$  & $A$ & $X$\\
       \hline
       $\{1,4\}$  & $\begin{pmatrix}
           1 & 2 & 0 & 0 & 0 \\
 1 & 0 & 2 & 0 & 0 \\
 0 & 4 & 0 & 2 & 0 \\
 0 & 0 & -1 & -1 & 2 \\
 0 & 0 & 0 & 1 & 0 
       \end{pmatrix}$  &
       $\begin{pmatrix}
           0 & 0 & 2 & 1 & 1 \\
 0 & 4 & 0 & -2 & 0 \\
 2 & 0 & 2 & 0 & -1 \\
 -2 & 2 & 0 & 0 & 0 \\
 -4 & 0 & 4 & 0 & 2 \\
       \end{pmatrix}$\\
       \hline
      $\{2,3\}$   & $\begin{pmatrix}
 0 & 1 & 0 & 0 & 0 \\
 2 & 1 & 1 & 0 & 0 \\
 0 & 2 & 2 & 1 & 0 \\
 0 & 0 & 1 & 0 & 1 \\
 0 & 0 & 0 & 1 & 0 \\
 \end{pmatrix}$ & $\begin{pmatrix}
 -2 & 0 & 4 & -8 & 4 \\
 0 & 0 & 0 & 4 & -4 \\
 1 & 0 & 0 & 0 & 2 \\
 -2 & 2 & 0 & -2 & 0 \\
 1 & -2 & 2 & 0 & -4 \\
\end{pmatrix}$\\
\hline
      $\{1,2,5\}$   & $\left(
\begin{array}{ccccc}
 1 & 1 & 0 & 0 & 0 \\
 \frac{7}{3} & 2 & 1 & 0 & 0 \\
 0 & 1 & 0 & 1 & 0 \\
 0 & 0 & 1 & 0 & 1 \\
 0 & 0 & 0 & 1 & 3 \\
\end{array}
\right)$ & $\left(
\begin{array}{ccccc}
 0 & 0 & 0 & 0 & 7 \\
 0 & 0 & 0 & 3 & 6 \\
 0 & 0 & 3 & 0 & 2 \\
 0 & 3 & 0 & 2 & 0 \\
 3 & 6 & 2 & 0 & 0 \\
\end{array}
\right)$ \\
\hline
      $\{1,3,4\}$   &$\left(
\begin{array}{ccccc}
 1 & 1 & 0 & 0 & 0 \\
 10 & 0 & 1 & 0 & 0 \\
 0 & 2 & -\frac{9}{35} & 1 & 0 \\
 0 & 0 & \frac{10296}{1225} & \frac{46}{35} & 1 \\
 0 & 0 & 0 & 1 & 0 \\
\end{array}
\right)$  &$\footnotesize\left(
\begin{array}{ccccc}
 0 & 0 & 10296 & \frac{453024}{35} & -\frac{41184}{5} \\
 0 & \frac{5148}{5} & 0 & \frac{41184}{5} & \frac{370656}{175} \\
 \frac{2574}{5} & 0 & 0 & 0 & \frac{41184}{5} \\
 77 & 490 & 0 & 0 & 0 \\
 -49 & 126 & 980 & 0 & -\frac{41184}{5} \\
\end{array}
\right)$ \\
\hline
      $\{1,3,5\}$   & $ \footnotesize\left(
\begin{array}{ccccc}
 1 & 1 & 0 & 0 & 0 \\
 \frac{3}{4} \left(2+\sqrt{3}\right) & 0 & 1 & 0 & 0 \\
 0 & \frac{1}{4} \left(4+\sqrt{3}\right) & 1 & 1 & 0 \\
 0 & 0 & 2 & 0 & 1 \\
 0 & 0 & 0 & 3 & 1 \\
\end{array}
\right)$ & $\scalebox{0.7}{$
\left(
\begin{array}{ccccc}
 0 & 0 & \frac{3}{4} (2+\sqrt{3}) & 0 & -\frac{3}{8} (3+\sqrt{3}) \\
 0 & 1 & 0 & \frac{1}{2} (1+\sqrt{3}) & 0 \\
 -\frac{4}{13} (\sqrt{3}-4) & 0 & 0 & 0 & \frac{3}{2} (1+\sqrt{3}) \\
 0 & \frac{1}{13} (1+3 \sqrt{3}) & 0 & \frac{1}{2} (1+\sqrt{3}) & 0 \\
 \frac{1}{39} (7 \sqrt{3}-15) & 0 & \frac{1}{4} (1+\sqrt{3}) & 0 & 0
\end{array}
\right)
$}$\\
    \end{tabular}
    \caption{Matrices $A\in{\mathcal M}(\dvec{P}_{5,\loops})$ and $X$ with $A\circ X=[A,X\trans]=O$.}
    \label{tab:n=5-missing-loops}
\end{table}

For other assignments of $\loops$ we present in Table~\ref{tab:n=5-missing-loops} matrices $A\in\mathcal{M}(\dvec{P}_{5,\loops})$ and $X$, such that $A\circ X=[A,X\trans]=O$, which shows $\dvec{P}_{5,\loops}$ does not require the nSSP. However, they are all supergraphs of $\dvec{P}_{5,\{1\}}$ or $\dvec{P}_{5,\{3\}}$, and so by Lemma~\ref{lem:superpattern-lemma} (and either Theorem~\ref{thm:double_paths-first-m-loops} or Example~\ref{ex:central_being_looped}) they all allow the nSSP. With this we characterized all non-isomorphic double paths on at most five vertices, which are presented in  Table~\ref{tab:small-paths}. 
\end{example}

\bibliographystyle{plain}
\bibliography{bibliography}
\end{document}